\documentclass[preprint,3p,times]{elsarticle}
\usepackage[toc,page,title,titletoc,header]{appendix}
\usepackage{geometry}
\usepackage{algorithm}
\usepackage{algorithmic}
\usepackage{amsmath}
\usepackage{amssymb}
\usepackage{amsthm}
\usepackage{enumerate}
\usepackage{enumitem}
\usepackage{mathrsfs}
\usepackage{graphicx}
\usepackage{epsfig}
\newtheorem{example}{Example}
\usepackage{tikz,pgfplots,tikz-3dplot}
\usepackage{epstopdf}
\usepackage{microtype}
\usepackage{natbib}
\usetikzlibrary{fit}
\biboptions{sort&compress}

\newtheorem{thm}{Theorem}[section]
\newtheorem{lemma}{Lemma}[section]

\newtheorem{rem}{Remark}
\renewcommand{}

\usepackage{cases}

\def\epsilon{\varepsilon} 
\newcommand{\mat}[1]{\boldsymbol{#1}}

\allowdisplaybreaks
\begin{document}
	\begin{frontmatter}
		\title{
			Adaptive moving mesh methods for the planar Willmore flow}
		
		\author[1]{Zhenghua Duan}
		\author[2]{Meng Li*}
		\address[1]{School of Mathematics and Statistics, Zhengzhou University, Zhengzhou 450001, China}
		\address[2]{School of Mathematics and Statistics, Huazhong University of Science and Technology, Wuhan 430074, China}

		\ead{This author's research was supported by the National Natural Science Foundation of Henan (No. 262300421876) and the National Natural Science Foundation of China (No. 11801527). Corresponding author: Meng\_Li@hust.edu.cn}
		
		\begin{abstract}
			In this paper, we propose adaptive moving mesh methods for the planar Willmore flow by incorporating a tangential velocity into the original geometric evolution. The tangential velocity is designed based on a monitor function constructed from the curvature and its variation, enabling dynamic mesh redistribution along the evolving interface. This adaptive redistribution enhances spatial resolution in regions of high geometric complexity while preserving mesh regularity. 
			The resulting moving mesh formulation is discretized using the $k$th-order backward differentiation formula (BDF$k$) in time and finite difference methods in space.
			Moreover, a class of new relaxed Lagrange multiplier approaches is further incorporated into the adaptive moving mesh framework to construct energy-stable adaptive moving mesh methods. 
			Furthermore, to enhance the adaptivity and flexibility of the proposed framework, we develop an adaptive strategy for selecting the monitor function and introduce an alternative redistribution approach.
			Finally, extensive numerical experiments demonstrate that the proposed BDF$k$-based adaptive scheme accurately captures the geometric evolution of the planar Willmore flow and exhibits excellent robustness and computational efficiency for problems involving complex interface geometries. 
		\end{abstract}

	\end{frontmatter}
	\section{Introduction}\label{sec:intro}\numberwithin{figure}{section}
	\numberwithin{equation}{section}
	Curvature-driven geometric evolution has been widely studied in geometric analysis, computational geometry, and physical modeling 
	\cite{Willmore1993,Helfrich1973,Deckelnick02}. 
	In many applications, the evolution of a curve is governed by the minimization of curvature-dependent bending energy, which drives the curve toward smoother and energetically favorable configurations \cite{Helfrich1973}. 
	The Willmore energy, originally introduced in differential geometry, provides a fundamental curvature-based energy model for planar curve evolution. Its planar curve formulation has been extensively studied and serves as an important framework for curve smoothing, shape optimization, and geometric evolution problems 
	\cite{Deckelnick02,Dziuk2008,Barrett08}.

	The planar Willmore flow is the $L^2$-gradient flow of the bending energy \cite{Barrett08,Dziuk2002}. 
	For a planar curve $\Gamma$, the Willmore energy is defined by
	\begin{align} \label{willmore_W}
		W(\Gamma)=\frac{1}{2}\int_\Gamma \kappa^2 ds,
	\end{align}
	where $\kappa$ denotes the curvature and $ds$ represents the arclength element. 
	Let $\mathbf{X}(s,t)$ be a parametrization of the evolving curve $\Gamma(t)$. 
	The normal velocity of the Willmore flow is given by
	\cite{BARRETT2020275}
	\begin{align*}
		\partial_t\mathbf{X}\cdot\mathbf{n}
		=V,
		\qquad
		V=\partial_{ss}\kappa+\frac{1}{2}\kappa^3 ,
	\end{align*}
	where $V$ denotes the normal velocity and $\partial_s$ represents differentiation with respect to arclength.
	As a fourth-order geometric evolution equation, the planar Willmore flow involves high-order curvature derivatives, which introduce severe stiffness and bring challenges to stable and efficient numerical discretization. Moreover, maintaining a high-quality mesh distribution during the evolution is crucial for accurately resolving regions with significant curvature variation.
	
	Several numerical approaches have been developed for the Willmore flow, particularly for its planar curve formulation.
	For planar curves, Rusu \cite{Rusu2005} formulated a mixed weak problem by introducing curvature as an additional unknown and constructed a finite element approximation for the Willmore flow.
	Dziuk \cite{DeckelnickDziuk2006,Dziuk2008} developed finite element approximations for Willmore flows based on variational formulations and auxiliary variables, providing important foundations for the numerical treatment of fourth-order geometric evolution equations. 
	Barrett, Garcke, and N\"urnberg proposed a parametric finite element formulation (commonly referred to as the BGN method) that avoids explicit mesh redistribution procedures or remeshing. Specifically, the weak formulation naturally generates a tangential velocity that drives mesh nodes toward an approximately equidistributed configuration along the curve, thereby reducing mesh distortion during the evolution \cite{Barrett07,Barrett08JCP,barrett2008parametric}. 
	More recently, Bao et al. \cite{Bao2025} developed a fully discrete parametric finite element method (PFEM) for the planar Willmore flow and proved unconditional energy stability based on newly derived geometric identities. 
	Garcke et al. proposed an energy-stable PFEM for Willmore flow based on the normal-tangential velocity splitting approach. The tangential velocity improves mesh quality without remeshing; furthermore, in the axisymmetric case, the curvature of the generating curve is used as a Lagrange multiplier to achieve discrete equidistribution of mesh points \cite{GarckeNurnbergZhao2020,GarckeNurnbergZhao2025,GarckeNurnbergZhao2026}.
	For Willmore flow problems, previous studies have primarily focused on two key aspects: preserving energy stability and maintaining mesh quality. In particular, the BGN method exploits the tangential motion naturally induced by its discrete formulation, enabling mesh points to evolve toward an approximately equidistributed configuration and thereby maintaining a high-quality parametrization throughout the evolution. In addition to the BGN method, Elliott and Fritz proposed a novel reparametrization strategy based on special solutions of the harmonic map heat flow. By incorporating an intrinsic tangential velocity via the DeTurck trick, their approach effectively redistributes mesh points and significantly improves mesh quality during geometric evolution \cite{m2017approximations}.
	Jiang et al.~\cite{Jiang2026Structure} proposed an energy-stable PFEM for planar elastic flows, where a Lagrange multiplier formulation was introduced to preserve the energy dissipation law and an explicit tangential velocity was employed to improve mesh quality.
	Although these approaches have demonstrated satisfactory mesh quality for various geometric evolution problems, their performance may deteriorate when the evolving curve develops regions with large curvature or strongly localized geometric variations. In such cases, adaptive moving mesh methods and mesh redistribution strategies provide effective approaches for dynamically controlling the spatial distribution of mesh points.
	
	A broad class of adaptive moving mesh methods based on the equidistribution principle has been developed \cite{budd2009adaptivity,Huang2011Adaptive,huang1994moving}. Huang et al. \cite{huang1994moving} investigated moving mesh partial differential equations derived from the equidistribution principle and analyzed their stability and mesh non-crossing properties. Li et al. \cite{LI2001562} proposed an adaptive moving mesh framework based on harmonic maps, where mesh redistribution is achieved by constructing a harmonic mapping between the physical and computational domains while keeping the number of mesh nodes unchanged. Mackenzie et al. \cite{Mackenzie2019} further developed an adaptive moving mesh method for the forced mean curvature flow, in which the tangential velocity is determined through an equidistribution principle and enables effective mesh concentration in regions with large curvature.
	Although adaptive moving mesh methods have been extensively studied for a variety of curvature-driven geometric evolution problems, those specifically designed for the Willmore flow remain relatively scarce. Consequently, developing adaptive moving mesh methods that can accommodate the fourth-order nature of the Willmore flow, maintain high-quality meshes under strong curvature variations, and remain robust in long-time simulations continues to be a significant challenge.
	
	In addition to maintaining good mesh quality, designing numerical discretizations for geometric flows that preserve the energy stability property remains a challenging problem. 
	The Lagrange multiplier approach provides an efficient structure-preserving framework for constructing energy-stable schemes while maintaining the original gradient flow structure. 
	Garcke et al.~\cite{garcke2025structure} developed novel structure-preserving PFEMs based on a Lagrange multiplier formulation, which preserve the geometric properties of mean curvature and surface diffusion flows at the fully discrete level. However, directly incorporating this Lagrange multiplier formulation into the adaptive moving mesh framework for the Willmore flow leads to ill-posed nonlinear systems, with the associated iterative solvers consistently failing to converge in our extensive numerical experiments. The new Lagrange multiplier approach proposed by Cheng et al.~\cite{Cheng2020NewLM} also encounters this difficulty when applied to the adaptive moving mesh formulation for the Willmore flow and is therefore not directly applicable in this setting.
	More recently, relaxed Lagrange multiplier approaches \cite{Jing2026RLM,Zhang2026RLMPF} have been developed for a variety of gradient flows and  more general dissipative systems. By introducing relaxation mechanisms, these approaches improve the solvability and robustness of numerical schemes while preserving the energy stability of the underlying gradient flows. Motivated by \cite{Zhang2026RLMPF}, we develop  energy-stable, adaptive moving mesh methods for the planar Willmore flow by incorporating the relaxed Lagrange multiplier approach into the adaptive moving mesh methods. The resulting methods not only maintain high-quality mesh distributions throughout the evolution but also admit a rigorous proof of energy stability at the fully discrete level, making them well suited for accurate and robust long-time simulations of the planar Willmore flow.
	
	The main contributions of this work are summarized as follows:
	\begin{itemize}
		
		\item 
		Based on the $k$th-order backward differentiation formula (BDF$k$) 
		in time and second-order centered finite differences in space, we construct a family of finite difference methods (FDMs) for the Willmore flow, referred to as BDF$k$-FDMs. These methods provide an efficient framework for the numerical approximation of the fourth-order geometric evolution equation. 
		However, for curves with complex geometric structures or strong curvature variations, maintaining mesh quality during long-time evolution remains a challenging issue.
		
		\item
		By incorporating a tangential velocity derived from the variational derivative of a mesh functional associated with a monitor function into the normal velocity equation, we establish a new total velocity formulation for the planar Willmore flow. The resulting formulation leads to an adaptive moving mesh system, in which mesh redistribution is intrinsically coupled with the geometric evolution rather than achieved through explicit reparametrization. 
		Here, the 
		monitor function, constructed from geometric quantities such as curvature 
		and curvature variation, enables the mesh to automatically adjust according 
		to local geometric features while preserving mesh regularity. 
		Based on this adaptive formulation, we construct adaptive BDF$k$-FDMs 
		(A-BDF$k$-FDMs) for the numerical solution of the resulting system.
		
		\item
		We incorporate a new type of relaxed Lagrange multiplier approach into the proposed adaptive moving mesh framework to enhance the energy stability of long-time simulations. Based on this formulation, we develop fully discrete energy-stable schemes, termed the A-RLM-BDF$k$-FDMs ($k=1,2$). 
		\item  
		An adaptive monitor function selection strategy is developed to automatically determine suitable mesh indicators according to the geometric characteristics of the evolving curve. The proposed strategy can be seamlessly incorporated into the A-BDF$k$-FDMs, the A-LM-BDF$k$-FDMs, and the A-RLM-BDF$k$-FDMs without modifying their algorithmic framework. In addition, an alternative
		adaptive 
		 weighted arc-length redistribution (A-WAR) algorithm is proposed to further improve mesh quality and computational efficiency.
	\end{itemize}
	
	The remainder of the paper is organized as follows. Section \ref{sec2} presents the parametric formulation of the planar Willmore flow and develops the corresponding BDF$k$-FDMs. The mesh properties and limitations of these schemes are also discussed. Section \ref{sec3} introduces an adaptive moving mesh formulation based on a tangential velocity and develops the corresponding A-BDF$k$-FDMs. Section \ref{sec4} presents the relaxed Lagrange multiplier formulation and develops the corresponding energy-stable A-RLM-BDF$k$-FDMs.
	Numerical experiments are reported in Section \ref{sec5}, followed by 
	conclusions and future perspectives in Section \ref{sec6}. Appendix A provides the geometric motivation for the construction of the monitor functions, and Appendix B further presents the adaptive monitor selection strategy and the A-WAR strategy.

	\section{BDF$k$-FDMs}\label{sec2}
	
	In this section, we first recall the classical continuous formulation of the planar Willmore flow. Then, we develop its fully discrete FDMs, 
	based on temporal BDF$k$ discretizations. 
	
	\subsection{Continuous formulation} \label{model_notation}
	Let $\Gamma(t)\subset\mathbb{R}^2$ denote a time-dependent smooth planar curve.
	Throughout the evolution, we represent $\Gamma(t)$ by a parametrized map
	\begin{equation*}
		\mathbf{X}(\cdot,t):\mathbb{T}\to\mathbb{R}^2,
		\qquad
		\rho \mapsto \mathbf{X}(\rho,t)
		= \big(x(\rho,t),\, y(\rho,t)\big)^{T},
	\end{equation*}
	where $\mathbb{T}:=\mathbb{R}/\mathbb{Z}=[0,1]$ is the periodic reference domain.
	The initial curve is parametrized by $\mathbf{X}(\cdot,0)=\mathbf{X}^0(\cdot).$
	Based on this parametrization, the arc-length coordinate is introduced as $s(\rho,t)=\int_{0}^{\rho}|\partial_q\mathbf{X}(q,t)|\,dq.$ With this definition, differentiation with respect to arc length is written as $\partial_s = \frac{1}{|\partial_\rho\mathbf{X}|}\,\partial_\rho$, and the corresponding arc-length element becomes $ds = |\partial_\rho\mathbf{X}|\,d\rho$. In addition, the tangent and normal vectors associated with the curve $\Gamma(t)$ are defined as follows
	\begin{equation} \label{eq:normal}
		\boldsymbol{\tau} := \boldsymbol{\tau}(\rho, t) = \partial_s \mathbf{X}(\rho, t) = \frac{\partial_\rho \mathbf{X}(\rho, t)}{|\partial_\rho \mathbf{X}(\rho, t)|}, \quad \mat{n} := \mat{\mat{n}}(\rho,t) = -\mat{\tau^{\bot}},
	\end{equation}
	where $\bot$ denotes the clockwise rotation by $\frac{\pi}{2}$.
	
	Then, the Willmore flow can be written as the following fourth-order geometric system. 
	Given $\mathbf{X}(\rho,0)=\mathbf{X}^0(\rho)$, we find 
	$\left(\mathbf{X}(\rho, t), V(\rho, t), \kappa(\rho, t)\right)$, $(\rho, t)\in\mathbb{T}\times (0, +\infty)$, such that 
	\begin{subequations}\label{WM}
		\begin{align}
			&\partial_t \mathbf{X} \cdot \mat{n} = V,\label{WM_a} \\[4pt]
			&V = \partial_{ss}\kappa + \frac{1}{2}\kappa^{3} = \frac{\partial_{\rho \rho} \kappa}{\left|\partial_{\rho}\mathbf{X}\right|^2} - \frac{\partial_{\rho} \kappa \,\partial_{\rho}\mathbf{X} \cdot \partial_{\rho \rho} \mathbf{X}}{\left|\partial_{\rho}\mathbf{X}\right|^4} + \frac{1}{2}\kappa^{3}, \label{WM_b}\\[4pt]
			&\kappa \,\mat{n} = -\,\partial_{ss}\mathbf{X} = -\frac{\partial_{\rho \rho} \mathbf{X}}{\left|\partial_{\rho}\mathbf{X}\right|^2} + \frac{\boldsymbol{\tau} \cdot \partial_{\rho\rho}\mathbf{X}}{\left|\partial_{\rho}\mathbf{X}\right|^3}\partial_{\rho}\mathbf{X}.\label{WM_c}
		\end{align}    
	\end{subequations}
	\begin{rem}
		Since the continuous formulation \eqref{WM} coincides with that of the BGN framework, the resulting FDMs naturally yield approximately equidistributed meshes. The sole distinction between the present work and the classical BGN methods lies in the adoption of FDMs rather than FEMs.
		In the next subsection, we construct the corresponding FDMs and examine their mesh-preserving properties through simple numerical experiments. We also discuss several issues that may arise in more complex settings. 
	\end{rem}
	
	\subsection{Finite difference discretizations}\label{BDFk_FDM}
	In this subsection, we present several finite difference discretizations of the Willmore flow \eqref{WM}.
	Let $h := \frac{1}{M}$ and $\Delta t := \frac{T}{N}$ denote the spatial grid size and the time step, where $M,N \in \mathbb{N}$. We introduce the uniform spatial partition
	$ \rho_i := ih$, for $i=0,1,\dots,M-1,$ with the periodicity condition $\rho_M\equiv\rho_0$, and the temporal grid $t^n := n\Delta t, n=0,1,2,\dots,N$. 
	The discrete curve at time $t^n$ is then given by the nodal positions
	$\mathbf{X}_i^n$, which approximate $\mathbf{X}(\rho_i,t^n) \in \mathbb{R}^2$, for $i = 0,1,\dots,M-1$ and $n=0,1,2,\dots,N$.
	
	To prepare the finite difference discretizations, we introduce the discrete
	derivative operators along the periodic parameter~$\rho$.
	Let $f(\rho)$ denote a scalar- or vector-valued function sampled at the
	uniform grid points $\rho_i$ with periodicity $f(\rho_{i+M})=f(\rho_i)$.
	The forward, backward, centered first-order differences, together with the
	standard centered second-order difference used for curvature and other higher-order geometric quantities, are given by
	\begin{align*}
		\delta_{\rho}^+ f(\rho_i) &= \frac{f(\rho_i+h)-f(\rho_i)}{h}, \qquad 
		\delta_{\rho}^- f(\rho_i) = \frac{f(\rho_i)-f(\rho_i-h)}{h}, \\[4pt]
		\delta_{\rho} f(\rho_i)
		&= \frac{f(\rho_i+h)-f(\rho_i-h)}{2h}, \qquad
		\delta_{\rho\rho} f(\rho_i)
		= \frac{f(\rho_i+h)-2f(\rho_i)+f(\rho_i-h)}{h^2},
	\end{align*}
	where periodic extension is assumed, and all operators act componentwise when $f$ is vector-valued.
	In the discrete settings, all geometric quantities introduced in Section~\ref{model_notation}
	are approximated using finite difference operators along the periodic
	parameter~$\rho$. In particular, the discrete tangent derivative is defined
	by the centered difference
	\begin{align*}
		\delta_\rho \mathbf{X}_i^n :=
		\frac{\mathbf{X}_{i+1}^n - \mathbf{X}_{i-1}^n}{2h}\approx \partial_\rho \mathbf{X}(\rho_i,t^n),    
	\end{align*}
	and then the unit tangent and normal vectors are defined by 
	\begin{align*}
		\boldsymbol{\tau}_i^n
		:=\frac{\delta_\rho \mathbf{X}_i^n}{\lvert\delta_\rho \mathbf{X}_i^n\rvert},
		\qquad
		\boldsymbol{n}_i^n := R\,\boldsymbol{\tau}_i^n,
		\quad
		R=\begin{pmatrix}0&-1\\[2pt]1&0\end{pmatrix}.  
	\end{align*}
	
	We discretize the spatial derivatives using second-order centered finite 
	differences, and approximate the temporal derivative by using the BDF$k$ method. 
	Since the BDF$k$ method requires the solution values from the previous $k-1$ 
	time levels, the time integration is initialized by lower-order BDF schemes.
	Combining these approximations yields a family of BDF$k$ finite difference approximations for the parametric Willmore flow~\eqref{WM}, hereafter referred to 
	as the BDF$k$-FDMs. For given 
	$\left\{\left(\mathbf{X}_i^{\,n-p}, V_i^{\,n-p}, \kappa_i^{\,n-p}\right)\right\}_{p=0}^{k-1}$, the BDF$k$-FDM seeks the updated solution $\left(\mathbf{X}_i^{\,n+1}, V_i^{\,n+1}, \kappa_i^{\,n+1}\right)$, 
	such that
	\begin{subequations}\label{FD-BDFk}
		\begin{align}
			&\frac{\sum_{p=0}^{k}\alpha_p \mathbf{X}_i^{n+1-p}}{\Delta t}\cdot 
			\mat{n}_i^{\,n+1}
			= V_i^{\,n+1}, \\[4pt]
			&V_i^{\,n+1} = 
			\frac{\delta_{\rho\rho}\kappa_i^{\,n+1}}{\left|\delta_{\rho}\mathbf{X}_i^{\,n+1}\right|^2}
			-
			\frac{\delta_{\rho}\kappa_i^{\,n+1}\,
				\delta_{\rho}\mathbf{X}_i^{\,n+1}\cdot
				\delta_{\rho\rho}\mathbf{X}_i^{\,n+1}}
			{\left|\delta_{\rho}\mathbf{X}_i^{\,n+1}\right|^4}
			+ \frac12 \left(\kappa_i^{\,n+1}\right)^3, \\[4pt]
			&\kappa_i^{\,n+1}\mat{n}_i^{\,n+1}
			= -\frac{\delta_{\rho\rho}\mathbf{X}_i^{\,n+1}}{\left|\delta_{\rho}\mathbf{X}_i^{\,n+1}\right|^2}
			+
			\frac{\boldsymbol{\tau}_i^{\,n+1}\cdot 
				\delta_{\rho\rho}\mathbf{X}_i^{\,n+1}}
			{\left|\delta_{\rho}\mathbf{X}_i^{\,n+1}\right|^{3}}
			\,\delta_{\rho}\mathbf{X}_i^{\,n+1},
		\end{align}
	\end{subequations}
	where the coefficients $\{\alpha_p\}_{p=0}^k$ correspond to the $k$-step BDF$k$ method.  
	For instance, the coefficients of several commonly used BDF$k$ methods are denoted by 
	\begin{align*}
		&\text{BDF1 }(k=1): 
		&& \alpha_0 = 1,\quad \alpha_1 = -1; \\[4pt]
		&\text{BDF2 }(k=2): 
		&& \alpha_0 = \tfrac{3}{2},\quad \alpha_1 = -2,\quad \alpha_2 = \tfrac{1}{2}; \\[4pt]
		&\text{BDF3 }(k=3): 
		&& \alpha_0 = \tfrac{11}{6},\quad \alpha_1 = -3,\quad \alpha_2 = \tfrac{3}{2},\quad \alpha_3 = -\tfrac{1}{3}; \\[4pt]
		&\text{BDF4 }(k=4): 
		&& \alpha_0 = \tfrac{25}{12},\quad \alpha_1 = -4,\quad \alpha_2 = 3,\quad 
		\alpha_3 = -\tfrac{4}{3},\quad \alpha_4 = \tfrac{1}{4}. 
	\end{align*}
	
	To treat the nonlinearity in \eqref{FD-BDFk}, we adopt a Picard-type 
	iterative strategy at each time step. 
	Given the iteration $\left(\mathbf{X}_i^{\,n+1,m},\, V_i^{\,n+1,m},\, \kappa_i^{\,n+1,m}\right)$, 
	the next iteration 
	$\left(\mathbf{X}_i^{\,n+1,m+1},\, V_i^{\,n+1,m+1},\, \kappa_i^{\,n+1,m+1}\right)$
	is obtained by solving
	\begin{subequations}
		\begin{align}
			&\frac{\alpha_0 \mathbf{X}_i^{\,n+1,m+1} +\sum_{p=1}^k \alpha_p \mathbf{X}_i^{\,n+1-p}}{\Delta t}
			\cdot \mat{n}_i^{\,n+1,m}
			= V_i^{\,n+1,m+1},  \\
			&V_i^{\,n+1,m+1}
			= \frac{\delta_{\rho\rho}\kappa_i^{\,n+1,m+1}}
			{\left|\delta_\rho \mathbf{X}_i^{\,n+1,m}\right|^2}
			- \frac{\delta_\rho \kappa_i^{\,n+1,m}\,
				\delta_\rho\mathbf{X}_i^{\,n+1,m}
				\cdot 
				\delta_{\rho\rho}\mathbf{X}_i^{\,n+1,m+1}}
			{\left|\delta_\rho\mathbf{X}_i^{\,n+1,m}\right|^4}
			+ \frac{1}{2}\big(\kappa_i^{\,n+1,m}\big)^2\,
			\kappa_i^{\,n+1,m+1},\\[4pt]
			&\kappa_i^{\,n+1,m+1}\,\mat{n}_i^{\,n+1,m}
			= -\frac{\delta_{\rho\rho}\mathbf{X}_i^{\,n+1,m+1}}
			{\left|\delta_\rho\mathbf{X}_i^{\,n+1,m}\right|^2}
			+ \frac{\boldsymbol{\tau}_i^{\,n+1,m}
				\cdot 
				\delta_{\rho\rho}\mathbf{X}_i^{\,n+1,m+1}}
			{\left|\delta_\rho\mathbf{X}_i^{\,n+1,m}\right|^3}
			\delta_\rho\mathbf{X}_i^{\,n+1,m+1}.
		\end{align}
	\end{subequations}
	The iteration is terminated once
	\[
	\max_i\!
	\left(
	\left\|\mathbf{X}_i^{\,n+1,m+1}-\mathbf{X}_i^{\,n+1,m}\right\|
	+ \left|V_i^{\,n+1,m+1}-V_i^{\,n+1,m}\right|
	+ \left|\kappa_i^{\,n+1,m+1}-\kappa_i^{\,n+1,m}\right|
	\right)
	< \varepsilon_{\mathrm{tol}},
	\]
	where $\varepsilon_{\mathrm{tol}}$ is a prescribed tolerance.  
	Upon convergence, we set
	\[
	\left(\mathbf{X}_i^{\,n+1},\,V_i^{\,n+1},\,\kappa_i^{\,n+1}\right)
	= \left(\mathbf{X}_i^{\,n+1,m+1},\,V_i^{\,n+1,m+1},\,\kappa_i^{\,n+1,m+1}\right). 
	\]
	
	\begin{rem} \label{BDFkmesh}
		Numerical results shown in Figs.~\ref{fig:BGN_good}–\ref{fig:BGN_bad} indicate that the BDF$k$-FDM generally maintains high mesh quality and produces stable and reliable results for nearly convex or mildly perturbed curves with smooth geometric variation, as illustrated in Fig.~\ref{fig:BGN_good}. However, when the evolving curve contains more intricate geometric features or undergoes rapid curvature changes, as in the examples of Fig.~\ref{fig:BGN_bad}, the mesh may gradually develop noticeable stretching, clustering, or other forms of nonuniformity during the evolution. For strongly oscillatory interfaces with sharp tips or drastic curvature transitions, the geometric quantities become highly sensitive from the very beginning, making the discrete system strongly nonlinear and stiff; in such cases, pronounced mesh tangling and even numerical breakdown may occur. These observations highlight the necessity of incorporating more effective mesh redistribution strategies in order to accurately capture the evolution of geometrically complex interfaces.
		\begin{figure}[!t]
			\centering
			\includegraphics[width=0.292\linewidth]{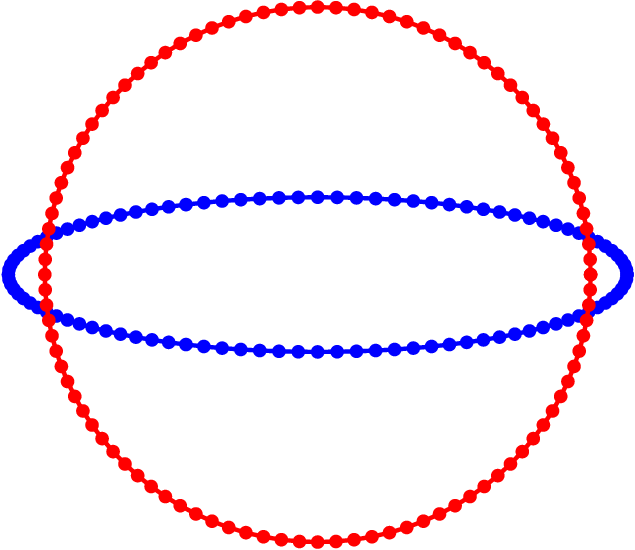}
			\hspace{0.12\textwidth}  
			\includegraphics[width=0.26\linewidth]{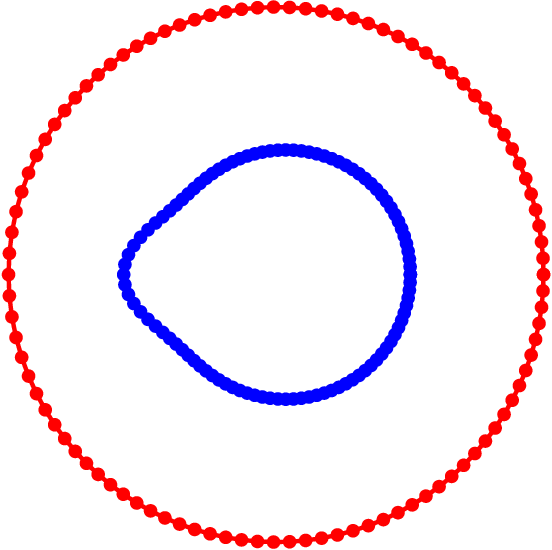}
			\caption{Representative examples showing well-maintained mesh quality. Blue curves represent the initial curves, and red curves represent the evolved curves at the final time. Left: the initial curves are given by $x = 4\cos(2\pi\rho), y = \sin(2\pi\rho)$, simulated with $k=1$, $M=100$, $T=50$, $\Delta t=0.01$. Right: the initial curves are given by $x = 4 + \left[1 + 0.3\exp\left(-\frac{(2\pi\rho - \pi)^{2}}{0.16}\right)\right]\cos (2\pi\rho), y = \left[1 + 0.3\exp\left(-\frac{(2\pi\rho - \pi)^{2}}{0.16}\right)\right]\sin (2\pi\rho)$, simulated with $k=1$, $M=100$, $T=10$, $\Delta t=0.01$.}
			\label{fig:BGN_good}
		\end{figure}
		\begin{figure}[!t]
			\centering
			\includegraphics[width=0.26\linewidth]{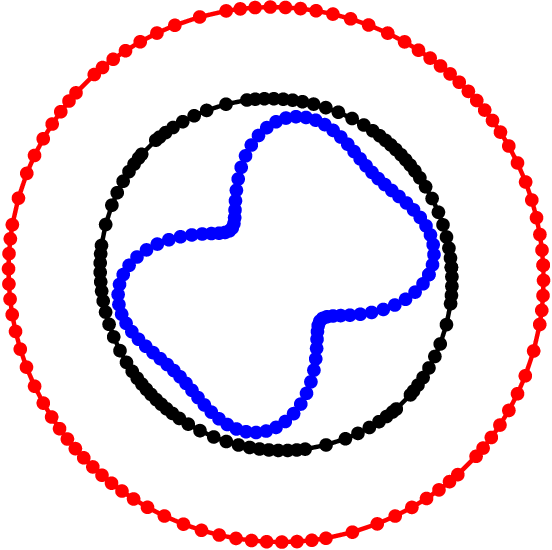}
			\hspace{0.05\textwidth} 
			\includegraphics[width=0.26\linewidth]{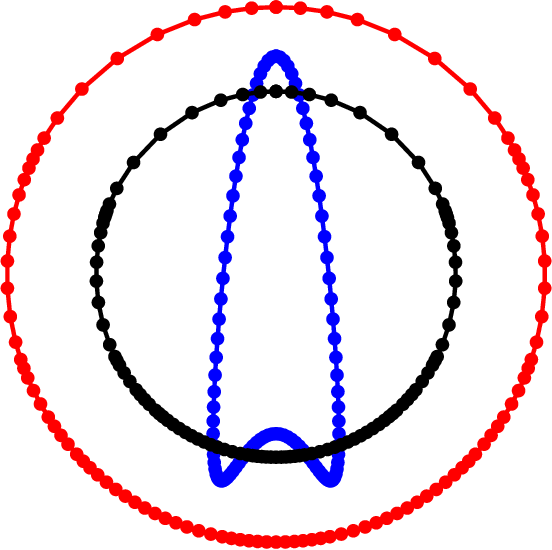}
			\hspace{0.05\textwidth} 
			\includegraphics[width=0.265\linewidth]{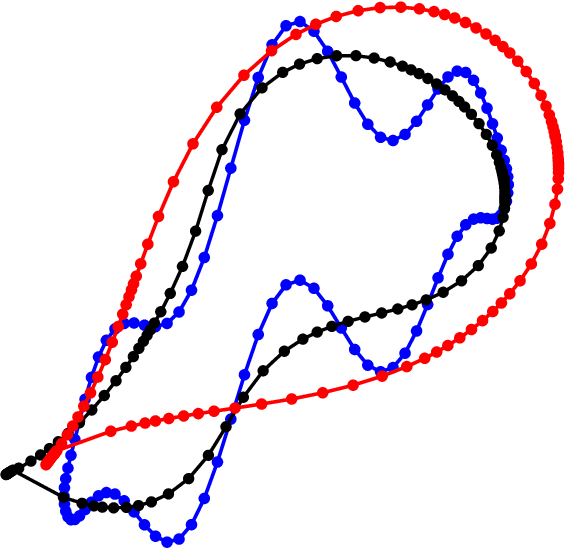}
			\caption{Representative examples showing mesh degradation under complex interface geometries. The blue, black, and red curves represent the initial curve, the evolved curve at the mid-time, and the final curve, respectively. Left: the initial curves are given by $x = \cos(2\pi\rho)\left[1 + 0.3\sin(4\pi\rho) + 0.2\cos(8\pi\rho)\right], y = \sin(2\pi\rho)\left[1 + 0.3\sin(4\pi\rho) + 0.2\cos(8\pi\rho)\right]$, simulated with $k=1$, $M=100$, $T=10$, $\Delta t=0.01$. Middle: the initial curves are given by $x = 0.5\sin(2\pi\rho), y = 1.5\cos(2\pi\rho)\left[1+\cos(2\pi\rho)\right]$, simulated with $k=1$,$M=100$, $T=10$, $\Delta t=0.01$. Right: the initial curves are given by $x = 1.2 \cos\left( 2\pi \rho \right), y = 0.5 \sin\left( 2\pi \rho \right)+ \sin\left( \cos\left( 2\pi \rho \right) \right)
				+ \sin\left( 2\pi \rho \right)\left[0.2 + \sin\left( 2\pi \rho \right)\sin^2\left( 6\pi\rho \right)\right]$, simulated with $k=1$, $M=100$, $T=1$, $\Delta t=0.0001$.} 
			\label{fig:BGN_bad}
		\end{figure}
		
		These observations highlight the necessity of developing more effective mesh redistribution strategies to accurately resolve geometrically complex evolutions. Motivated by this, we introduce in the following section an adaptive moving mesh formulation based on a tangential velocity, where the mesh redistribution is intrinsically coupled with the geometric evolution rather than performed through an additional remeshing procedure.
		
		
	\end{rem}

	\section{A-BDF$k$-FDMs}\label{sec3}
	In order to improve mesh quality during the evolution, mesh adaptivity can be
	achieved intrinsically by incorporating a suitable tangential velocity into
	the evolution equation. Such a tangential component does not alter the
	geometric motion of the evolving curve, but allows the parametrization to
	adjust continuously according to the local geometric features. Consequently,
	the mesh redistribution is intrinsically coupled with the geometric evolution without requiring an additional remeshing procedure. In this section, we introduce a tangential velocity into the Willmore flow formulation and develop an adaptive moving mesh framework, leading to the corresponding A-BDF$k$-FDMs.
	
	\subsection{The new coupled system}
	Based on the equidistribution principle, we derive an explicit expression for the tangential velocity component that drives the mesh toward the desired adaptive distribution while preserving the normal evolution law. The underlying idea is that the parametrization is allowed to evolve in the tangential direction, whereas the normal component is determined solely by the geometric evolution law.
	
	We aim to redistribute mesh points along the evolving curve \(\mathbf{X}(\rho,t)\) according to a prescribed monitor function \(M_f(\kappa, \kappa_s, t) > 0\). Let \(s(\rho,t)\) denote the arc-length coordinate, as defined in Section~\ref{model_notation}. From the definition of \(s\), we have the relationship
	\[
	\partial_\rho s = \left|\partial_\rho \mathbf{X}\right|.
	\]
	The equidistribution principle requires the weighted arc-length element \(M_f \, \mathrm{d}s\) to be uniformly distributed in \(\rho\). In other words, there exists a time-dependent constant \(C(t)\) such that
	\[
	M_f(\kappa, \kappa_s, t) \, \partial_\rho s(\rho,t) = C(t).
	\]
	By integrating this equation over the reference interval \([0,1]\), we obtain
	\[
	C(t) = \int_0^1 M_f(\kappa, \kappa_s, t) \, \partial_\rho s \, \mathrm{d}\rho = \int_{\Gamma(t)} M_f(\kappa, \kappa_s, t) \, \mathrm{d}s.
	\]
	Differentiating this equation with respect to \(\rho\) gives the differential form of the equidistribution condition:
	\[
	\partial_\rho\left(M_f(\kappa, \kappa_s, t) \, \partial_\rho s\right) = 0.
	\]
	
	To derive the tangential velocity that enforces this condition, we introduce the inverse parametrization \(\rho = \rho(s,t)\) and define the functional
	\[
	\mathcal{I}[\rho] = \frac{1}{2} \int_{\Gamma(t)} \frac{1}{M_f} \, \left(\partial_s \rho\right)^2 \, \mathrm{d}s.
	\]
	This choice of functional is motivated by the observation that if \(\partial_s \rho\) is proportional to \(M_f\), the computational parameter will allocate more arc-length to regions where \(M_f\) is larger, thus increasing the mesh density in these regions, which is precisely the desired effect of equidistribution.

	\begin{lemma}\label{lem:dts}
		The evolution of the arc-length coordinate \(s\) required to enforce the equidistribution condition is given by
		\[
		\partial_t s = - \frac{1}{\partial_s \rho} \, \frac{P}{\mathcal{J}} \, \partial_s \left( \frac{1}{M_f} \, \partial_s \rho \right),
		\]
		where \(\mathcal{J} > 0\) is a relaxation time constant, and \(P\) is a positive operator that controls the smoothing properties of the redistribution process.
	\end{lemma}
	
	\begin{proof}
		To compute the variational derivative of the functional \(\mathcal{I}[\rho]\), we perturb \(\rho\) by \(\rho + \varepsilon \eta\), where \(\eta\) is an arbitrary variation that vanishes at the boundaries. Applying integration by parts, we obtain
		\[
		\frac{\mathrm{d}}{\mathrm{d}\varepsilon} \mathcal{I}[\rho + \varepsilon \eta] \Big|_{\varepsilon = 0} = \int_{\Gamma(t)} \frac{1}{M_f} (\partial_s \rho)(\partial_s \eta) \, \mathrm{d}s = - \int_{\Gamma(t)} \partial_s \left( \frac{1}{M_f} \, \partial_s \rho \right) \eta \, \mathrm{d}s.
		\]
		Since this identity holds for all admissible variations \(\eta\), we conclude that the variational derivative is:
		\[
		\frac{\delta \mathcal{I}}{\delta \rho} = - \partial_s \left( \frac{1}{M_f} \, \partial_s \rho \right).
		\]
		For stationary points of the functional \(\mathcal{I}\), the Euler-Lagrange equation must hold
		\[
		\partial_s \left( \frac{1}{M_f} \, \partial_s \rho \right) = 0,
		\]
		which corresponds to the differential equidistribution condition. 
		
		To drive the parametrization toward the equilibrium condition, we introduce a relaxation dynamics for the inverse mapping \(\rho(s,t)\). Specifically, by evolving \(\rho\) in the direction of the negative gradient of the functional \(\mathcal{I}\), we obtain 
		\begin{align}\label{x1}
			\partial_t \rho = - \frac{P}{\mathcal{J}} \, \frac{\delta \mathcal{I}}{\delta \rho} = \frac{P}{\mathcal{J}} \, \partial_s \left( \frac{1}{M_f} \, \partial_s \rho \right),
		\end{align}
		where \(\mathcal{J} > 0\) is a relaxation time constant, and \(P\) is a positive operator that controls the smoothing properties of the redistribution process. This evolution ensures that the functional \(\mathcal{I}\) decreases monotonically over time, and the parametrization asymptotically approaches the equidistribution condition.
		Then, by using the identity \(\rho = \rho(s(\rho,t),t)\), which is valid for all \((\rho,t)\), and differentiating with respect to \(s\) while keeping \(t\) fixed, we obtain
		\begin{align}\label{x2}
			\partial_\rho s = \frac{1}{\partial_s \rho}.
		\end{align}
		Differentiating \eqref{x2} with respect to \(t\), we derive 
		\begin{align}\label{x3}
			\partial_t s = - \frac{\partial_t \rho}{\partial_s \rho}.
		\end{align}
		Substituting \eqref{x1} into \eqref{x3} gives 
		\[
		\partial_t s = - \frac{1}{\partial_s \rho} \, \frac{P}{\mathcal{J}} \, \partial_s \left( \frac{1}{M_f} \, \partial_s \rho \right).
		\]
		Therefore, the evolution equation for the arc-length coordinate \(s\) is derived.
	\end{proof}
	
	Using Lemma \ref{lem:dts}, we derive the equation for the total velocity.
	\begin{lemma}\label{lem:dtX_tau}
		The total velocity of the curve evolution is given by
		\begin{align} \label{tangential_velocity}
			\partial_t \mathbf{X} 
			= V\mathbf{n} + \left[\frac{P}{\mathcal{J}} \, \left( M_f \, \left|\partial_\rho \mathbf{X}\right| \right)^{-2} \, \partial_\rho \left( M_f \, \left|\partial_\rho \mathbf{X}\right| \right)\right]\boldsymbol{\tau} .
		\end{align}
	\end{lemma}
	
	\begin{proof}
		Let \(Q := M_f \, \partial_\rho s\). Then, we have 
		\begin{align}\label{y1}
			\partial_s \left( \frac{1}{M_f} \, \partial_s \rho \right)=\partial_s\left(\frac{1}{Q}\right) = - \frac{\partial_s Q}{Q^2}.
		\end{align}
		From Lemma \ref{lem:dts}, we obtain 
		\begin{align}\label{y2}
			\partial_t s =\frac{1}{\partial_s \rho} \, \frac{P}{\mathcal{J}} \frac{\partial_s Q}{Q^2}=\frac{1}{\partial_s \rho} \, \frac{P}{\mathcal{J}} \frac{\partial_s Q}{\left(M_f \, \partial_\rho s\right)^2}= \frac{P}{\mathcal{J}} \, \frac{\partial_s Q}{M_f Q}.
		\end{align}
		Hence, by using 
		\begin{align*}
			\partial_s Q = \frac{M_f}{Q} \, \partial_\rho Q
		\end{align*}
		in \eqref{y2}, we have
		\begin{align}\label{lem:dtX_tau_pf1}
			\partial_t s =\frac{P}{\mathcal{J}} \frac{M_f}{Q} \cdot \frac{\partial_\rho Q}{M_f Q} =\frac{P}{\mathcal{J}} \frac{\partial_\rho Q}{Q^2}=\frac{P}{\mathcal{J}} \, \left( M_f \, \partial_\rho s \right)^{-2} \, \partial_\rho \left( M_f \, \partial_\rho s \right).
		\end{align}
		The curve evolution at fixed \(\rho\) is given by
		\[
		\partial_t \mathbf{X} \big|_{\rho} = \partial_t \mathbf{X} \big|_s + \boldsymbol{\tau} \, \partial_t s.
		\]
		Since \(\partial_t \mathbf{X} \big|_s\) contributes only to the normal motion, we obtain the tangential velocity:
		\begin{align}\label{lem:dtX_tau_pf2}
			\partial_t \mathbf{X} \cdot \boldsymbol{\tau} = \partial_t s.
		\end{align}
		By virtue of \eqref{lem:dtX_tau_pf1} and \eqref{lem:dtX_tau_pf2}, we get
		\[
		\partial_t \mathbf{X} \cdot \boldsymbol{\tau} = \frac{P}{\mathcal{J}} \, \left( M_f \, \partial_\rho s \right)^{-2} \, \partial_\rho \left( M_f \, \partial_\rho s \right)
		= \frac{P}{\mathcal{J}} \, \left( M_f \, \left|\partial_\rho \mathbf{X}\right| \right)^{-2} \, \partial_\rho \left( M_f \, \left|\partial_\rho \mathbf{X}\right| \right).
		\]
		Finally, since the normal component of the curve evolution is given by 
		\(\partial_t \mathbf{X} \cdot \mathbf{n} = V\) from \eqref{WM_a}, we conclude
		\begin{align*}
			\partial_t \mathbf{X} = \left(\partial_t \mathbf{X}\cdot \mathbf{n}\right)\mathbf{n} + \left(\partial_t \mathbf{X}\cdot \boldsymbol{\tau}\right)\boldsymbol{\tau} = V\mathbf{n} + \left[\frac{P}{\mathcal{J}} \, \left( M_f \, \left|\partial_\rho \mathbf{X}\right| \right)^{-2} \, \partial_\rho \left( M_f \, \left|\partial_\rho \mathbf{X}\right| \right)\right]\boldsymbol{\tau}.
		\end{align*}
		Therefore, we have completed the proof. 
	\end{proof}
	
	We couple the total velocity \eqref{tangential_velocity} with the original Willmore flow~\eqref{WM} to obtain the following adaptive moving mesh formulation.  
	Given the initial curve 
	$\mathbf{X}(\rho,0) = \mathbf{X}^0(\rho)$,
	we seek 
	$\left(\mathbf{X}(\rho,t),\, V(\rho,t),\, \kappa(\rho,t)\right), 
	(\rho,t)\in\mathbb{T}\times(0,+\infty)$,
	such that 
	\begin{subequations}\label{WM-adaptive}
		\begin{align}
			&\partial_t \mathbf{X} = V\mathbf{n} + \left[\frac{P}{\mathcal{J}} \, \left( M_f \, \left|\partial_\rho \mathbf{X}\right| \right)^{-2} \, \partial_\rho \left( M_f \, \left|\partial_\rho \mathbf{X}\right| \right)\right]\boldsymbol{\tau} , \label{WM-adaptive-a}\\[4pt] 
			&V = \frac{\partial_{\rho \rho} \kappa}{\left|\partial_{\rho}\mathbf{X}\right|^2} - \frac{\partial_{\rho} \kappa \,\partial_{\rho}\mathbf{X} \cdot \partial_{\rho \rho} \mathbf{X}}{\left|\partial_{\rho}\mathbf{X}\right|^4} + \frac{1}{2}\kappa^{3}, \\[4pt]
			&\kappa = -\frac{\partial_{\rho \rho} \mathbf{X}\cdot \mat{n}}{\left|\partial_{\rho}\mathbf{X}\right|^2}.
		\end{align}    
	\end{subequations}

	\begin{rem}
		Since the adaptive moving mesh Willmore system~\eqref{WM-adaptive} introduces an additional tangential velocity~\eqref{tangential_velocity}, 
		the third equation in~\eqref{WM-adaptive}, which defines the curvature $\kappa$, 
		needs to be reformulated accordingly. 
		This modification ensures consistency between the geometric definition of curvature and the new tangential motion, 
		allowing the system to maintain correct geometric evolution under the adaptive mesh redistribution. 
	\end{rem}

	\begin{rem}
		The tangential velocity equation has been extensively studied in the literature
		\cite{Huang2011Adaptive,HUANG20086532} and has recently been applied to the forced
		mean curvature flow \cite{Mackenzie2019}. In the present work, we extend this framework to the Willmore flow by employing a
		full velocity formulation. 
		In comparison with the forced mean curvature
		flow, the Willmore flow exhibits significantly more complex geometric structures
		and involves higher-order curvature terms, rendering its evolution substantially
		more intricate and challenging from both analytical and numerical perspectives.
		Furthermore, in order to more thoroughly investigate these additional
		complexities, we consider more sophisticated forms of the monitor function. In
		particular, beyond the curvature $\kappa$ itself, we also incorporate
		derivatives of curvature $\kappa_s$ and the squared curvature $\kappa^2$, into the monitor function to better capture localized geometric variations along the interface. Since different evolutionary stages may exhibit distinct geometric behaviors, a single monitor function may not always provide an optimal mesh distribution. Therefore, an adaptive monitor selection strategy is developed to automatically choose a suitable monitor function according to the current geometric characteristics of the curve. The detailed construction of the monitor functions and the adaptive selection procedure are presented in Appendix A.
		In addition, we compare our results with those obtained using BDF$k$-FDMs that do not incorporate the tangential velocity component. Through this comparison, we highlight the advantages of introducing a tangential velocity component, demonstrating its superior ability to preserve mesh quality, accurately capture flow evolution, and achieve higher numerical accuracy.
	\end{rem}
	
	\subsection{Finite difference discretizations}
	Prior to the spatial discretization, we slightly reformulate the curve evolution equation ~\eqref{WM-adaptive-a} to obtain a form that is more convenient for numerical implementation. Recalling that
	\[
	\partial_\rho \left|\partial_\rho \mathbf{X}\right| 
	= \frac{\partial_\rho \mathbf{X}\cdot \partial_{\rho\rho}\mathbf{X}}
	{\left|\partial_\rho \mathbf{X}\right|},
	\]
	we can expand~\eqref{WM-adaptive-a} and express it as
	\begin{equation}\label{eq:tangential_reform1}
		\partial_t \mathbf{X} = V\mathbf{n} + \left[\frac{P}{\mathcal{J}}
		\frac{\partial_{\rho} M_f\,\partial_\rho \mathbf{X}
			+ M_f\,\partial_{\rho\rho}\mathbf{X}}
		{\left(M_f\,\left|\partial_\rho \mathbf{X}\right|\right)^{2}}
		\cdot\boldsymbol{\tau}\right]\boldsymbol{\tau}.
	\end{equation}
	Using the definition of $\boldsymbol{\tau}$ given in \eqref{eq:normal}, the above relation can be reorganized into the equivalent form
	\begin{equation}\label{eq:tangential_reform2}
		\partial_t \mathbf{X} = V\mathbf{n} + \left[\left(\frac{P\,M_f}{\mathcal{J}\left(M_f\,\left|\partial_\rho \mathbf{X}\right|\right)^{2}}
		\,\partial_{\rho\rho}\mathbf{X}
		\right)\!\cdot\!\boldsymbol{\tau}
		+ \frac{P\,\partial_{\rho} M_f\,\left|\partial_\rho \mathbf{X}\right|}
		{\mathcal{J}\left(M_f\,\left|\partial_\rho \mathbf{X}\right|\right)^{2}}\right]\boldsymbol{\tau}.
	\end{equation}
	
	We next present the discrete formulations of the adaptive moving mesh Willmore system \eqref{WM-adaptive}. Following the same discretization strategy as in Section \ref{BDFk_FDM}, we employ second-order centered finite differences for spatial derivatives and the BDF$k$ for temporal discretization. 
	Let $\left\{\left(\mathbf{X}_i^{n-p}, V_i^{n-p}, \kappa_i^{n-p}\right)\right\}_{p=0}^{k-1}$ be the known solution values at previous time levels. 
	Then, we employ the following A-BDF$k$-FDMs to numerically solve the adaptive moving mesh Willmore system \eqref{WM-adaptive}:
	\begin{subequations}\label{eq:willmore_MM}
		\begin{align}
			&\frac{1}{\Delta t}\sum_{p=0}^{k}\alpha_p\,\mathbf{X}_i^{\,n+1-p} = V_i^{\,n+1} \mathbf{n}_i^{\,n+1} + \left[\left(\frac{P_i^{n+1}\,M_{f,i}^{n+1}}{\mathcal{J}\left(M_{f,i}^{n+1}\,\left|\delta_\rho \mathbf{X}_i^{n+1}\right|\right)^{2}}
			\,\delta_{\rho\rho}\mathbf{X}_i^{n+1}
			\right)\!\cdot\!\boldsymbol{\tau}_i^{n+1} + \frac{P_i^{n+1}\,\delta_{\rho} M_{f,i}^{n+1}\,\left|\delta_\rho \mathbf{X}_i^{n+1}\right|}
			{\mathcal{J}\left(M_{f,i}^{n+1}\,\left|\delta_\rho \mathbf{X}_i^{n+1}\right|\right)^{2}}\right]\boldsymbol{\tau}_i^{n+1}, \label{eq:willmore_MM_a}\\[4pt]
			&V_i^{\,n+1} = 
			\frac{\delta_{\rho\rho}\kappa_i^{\,n+1}}{\left|\delta_\rho \mathbf{X}_i^{\,n+1}\right|^2}
			- \frac{\delta_\rho \kappa_i^{\,n+1}\,\left(\delta_\rho \mathbf{X}_i^{\,n+1}\cdot \delta_{\rho\rho} \mathbf{X}_i^{\,n+1}\right)}
			{\left|\delta_\rho \mathbf{X}_i^{\,n+1}\right|^4}
			+ \tfrac{1}{2}\,(\kappa_i^{\,n+1})^3,\label{eq:willmore_MM_b}\\[4pt]
			&\kappa_i^{\,n+1} = 
			-\frac{\delta_{\rho\rho}\mathbf{X}_i^{\,n+1}\cdot\mathbf{n}_i^{\,n+1}}
			{\left|\delta_\rho \mathbf{X}_i^{\,n+1}\right|^2}.\label{eq:willmore_MM_c}
		\end{align}
	\end{subequations}
	Here, $\delta_\rho$ and $\delta_{\rho\rho}$ denote the second-order centered finite-difference operators for the first and second spatial derivatives, respectively, and  
	the coefficients $\alpha_p$ are the BDF$k$ weights that have been defined in Section~\ref{BDFk_FDM}. 
	
	To efficiently solve the A-BDF$k$-FDMs, we employ a Picard-type iterative
	procedure at every time step.  This approach updates the geometric variables
	through a sequence of linearized problems until convergence is achieved.
	Starting from the iteration
	$\left(\mathbf{X}_i^{\,n+1,m},\,V_i^{\,n+1,m},\,\kappa_i^{\,n+1,m}\right)$,
	the subsequent iteration
	$\left(\mathbf{X}_i^{\,n+1,m+1},\,V_i^{\,n+1,m+1},\,\kappa_i^{\,n+1,m+1}\right)$
	is determined by solving
	\begin{subequations}\label{eq:4.8}
		\begin{align}
			&\frac{\alpha_0 \mathbf{X}_i^{\,n+1,m+1}+\sum_{p=1}^k \alpha_p \mathbf{X}_i^{\,n+1-p}}{\Delta t} = \left[\left(\frac{P_i^{n+1,m}\,M_{f,i}^{n+1,m}}{\mathcal{J}\left(M_{f,i}^{n+1,m}\,\left|\delta_\rho \mathbf{X}_i^{n+1,m}\right|\right)^{2}}
			\,\delta_{\rho\rho}\mathbf{X}_i^{n+1,m+1} 
			\right)\!\cdot\!\boldsymbol{\tau}_i^{n+1,m}\right]\boldsymbol{\tau}_i^{n+1,m} \nonumber \\
			&\hspace{3cm} + \left(\frac{P_i^{n+1,m}\,\delta_{\rho} M_{f,i}^{n+1,m}\,\left|\delta_\rho \mathbf{X}_i^{n+1,m}\right|}
			{\mathcal{J}\left(M_{f,i}^{n+1,m}\,\left|\delta_\rho \mathbf{X}_i^{n+1,m}\right|\right)^{2}}\right)\boldsymbol{\tau}_i^{n+1,m} + V_i^{\,n+1,m+1} \mathbf{n}_i^{\,n+1,m}, \\[3pt]
			&V_i^{\,n+1,m+1} =
			\frac{\delta_{\rho\rho}\kappa_i^{\,n+1,m+1}}{\left|\delta_\rho \mathbf{X}_i^{\,n+1,m}\right|^2}
			-\frac{\delta_\rho \kappa_i^{\,n+1,m}\,
				\left(\delta_\rho \mathbf{X}_i^{\,n+1,m}\!\cdot\!\delta_{\rho\rho}\mathbf{X}_i^{\,n+1,m+1}\right)}
			{\left|\delta_\rho \mathbf{X}_i^{\,n+1,m}\right|^4}
			+\tfrac{1}{2}\left(\kappa_i^{\,n+1,m}\right)^2\kappa_i^{\,n+1,m+1}, \\[3pt]
			&\kappa_i^{\,n+1,m+1} =
			-\frac{\delta_{\rho\rho}\mathbf{X}_i^{\,n+1,m+1}\!\cdot\!\mathbf{n}_i^{\,n+1,m}}
			{\left|\delta_\rho \mathbf{X}_i^{\,n+1,m}\right|^2}.
		\end{align}
	\end{subequations}
	The iteration continues until the following stopping criterion is satisfied:
	\[
	\max_i\!\left(
	\left\|\mathbf{X}_i^{\,n+1,m+1}-\mathbf{X}_i^{\,n+1,m}\right\|
	+\left|V_i^{\,n+1,m+1}-V_i^{\,n+1,m}\right|
	+\left|\kappa_i^{\,n+1,m+1}-\kappa_i^{\,n+1,m}\right|\right)
	<\varepsilon_{\mathrm{tol}}.
	\]
	Once this tolerance criterion is satisfied, we set 
	\[
	\left(\mathbf{X}_i^{\,n+1},\,V_i^{\,n+1},\,\kappa_i^{\,n+1}\right)
	=\left(\mathbf{X}_i^{\,n+1,m+1},\,V_i^{\,n+1,m+1},\,\kappa_i^{\,n+1,m+1}\right).
	\]
	
		\begin{rem}
		All adaptive numerical schemes considered in this paper are based on fully implicit formulations. 
		Since iterative solvers are required, the computational cost is higher than that of linearized methods; 
		however, in practical computations the number of iterations typically remains at a relatively low level. 
		Moreover, due to the use of finite difference discretizations in space, the computational cost of each iteration is relatively small, 
		and consequently the overall computational complexity is not significantly increased. 
%
%
	\end{rem}
	
	\begin{rem}
		Although the proposed A-BDF$k$-FDMs effectively incorporate mesh adaptation into the Willmore flow evolution and provide improved mesh distribution properties for complex geometric evolutions, establishing a rigorous energy stability result for the resulting fully discrete schemes remains challenging. Therefore, it is desirable to further develop an energy-stable formulation that preserves the advantages of the adaptive framework.
		In the next section, we establish energy-stable adaptive moving mesh methods through a type of new relaxed Lagrange multiplier approaches. The resulting A-RLM-BDF$k$-FDMs are constructed within the adaptive moving mesh framework, and the methods are shown to satisfy a discrete energy stability law while retaining the adaptive mesh redistribution properties.
	\end{rem}

	\section{A-RLM-BDF$k$-FDMs}\label{sec4}
	In this section, we develop energy-stable A-RLM-BDF$k$-FDMs ($k=1, 2$) by incorporating a relaxed Lagrange multiplier approach into the proposed A-BDF$k$-FDMs.

	\subsection{Relaxed formulation}
	The Willmore flow can be interpreted as the gradient flow of the Willmore
	energy \eqref{willmore_W}. Obviously, the Willmore flow \eqref{WM} holds the following energy decay law \cite{Barrett2020Parametric,Pan2026EnergyStable}
	\begin{equation}
		\frac{dW}{dt}
		=
		-\int_{\Gamma}
		\left(\partial_{ss}\kappa+\frac12\kappa^3\right)\partial_t \mathbf{X} \cdot \mathbf{n}\,ds
		=
		-\int_{\Gamma} V^2\,ds
		\leq 0.
	\end{equation}
	Based on this gradient flow structure, we introduce an auxiliary multiplier $q(t)$ and reformulate the original Willmore flow system~\eqref{WM} into the following equivalent formulation:
	\begin{subequations}\label{WM-adaptive-LM}
		\begin{align}
			&\partial_t \mathbf{X} = V\mathbf{n} + \left[\frac{P}{\mathcal{J}} \, \left( M_f \, \left|\partial_\rho \mathbf{X}\right| \right)^{-2} \, \partial_\rho \left( M_f \, \left|\partial_\rho \mathbf{X}\right| \right)\right]\boldsymbol{\tau} , \label{WM-adaptive-a}\\[4pt] 
			&V = q(t)\left(\frac{\partial_{\rho \rho} \kappa}{\left|\partial_{\rho}\mathbf{X}\right|^2} - \frac{\partial_{\rho} \kappa \,\partial_{\rho}\mathbf{X} \cdot \partial_{\rho \rho} \mathbf{X}}{\left|\partial_{\rho}\mathbf{X}\right|^4} + \frac{1}{2}\kappa^{3}\right), \\[4pt]
			&\kappa = -\frac{\partial_{\rho \rho} \mathbf{X}\cdot \mat{n}}{\left|\partial_{\rho}\mathbf{X}\right|^2}, \\[4pt]
			& \frac{dW}{dt} = -q(t)\int_{0}^{1}
			\left(\frac{\partial_{\rho \rho} \kappa}{\left|\partial_{\rho}\mathbf{X}\right|^2} - \frac{\partial_{\rho} \kappa \,\partial_{\rho}\mathbf{X} \cdot \partial_{\rho \rho} \mathbf{X}}{\left|\partial_{\rho}\mathbf{X}\right|^4} + \frac{1}{2}\kappa^{3}\right)\partial_t \mathbf{X} \cdot \mathbf{n}\left|\partial_\rho \mathbf{X}\right|\,d\rho.
		\end{align}    
	\end{subequations}
	
	By prescribing the initial condition $q(0)=1$, the system
	\eqref{WM-adaptive-LM} is equivalent to the original adaptive Willmore flow
	system \eqref{WM-adaptive}, since $q(t)\equiv1$ for all $t\ge0$.
	Consequently, the relaxed Lagrange multiplier formulation
	\eqref{WM-adaptive-LM} preserves the original energy decay law of the
	adaptive Willmore flow.
	Indeed, substituting the first two equations of
	\eqref{WM-adaptive-LM} into the energy evolution equation yields
	\begin{equation}
		\begin{aligned}
			\frac{dW}{dt}
			&=
			-q(t)\int_{\Gamma}
			\left(
			\frac{\partial_{\rho \rho} \kappa}{|\partial_{\rho}\mathbf{X}|^2}
			-\frac{\partial_{\rho}\kappa\,\partial_{\rho}\mathbf{X}
				\cdot\partial_{\rho\rho}\mathbf{X}}
			{|\partial_{\rho}\mathbf{X}|^4}
			+\frac12\kappa^3
			\right)
			\partial_t \mathbf{X}\cdot\mathbf{n}\,ds=
			-\int_{\Gamma}V^2\,ds
			\le0.
		\end{aligned}
	\end{equation}
	Therefore, the introduced Lagrange multiplier $q(t)$ preserves the original
	energy decay structure of the adaptive Willmore flow.

	However, despite preserving the original energy dissipation law at the
	continuous level, the corresponding FDMs based on the Lagrange multiplier
	formulation~\eqref{WM-adaptive-LM} often give rise to singular algebraic
	systems in our numerical experiments.
	To overcome this difficulty, motivated by~\cite{Jing2026RLM,Zhang2026RLMPF}, we introduce an auxiliary relaxation equation to replace the original constraint with a relaxed counterpart. Specifically, the constraint imposed on the Lagrange multiplier is relaxed, allowing the multiplier to evolve dynamically during the geometric evolution. As a result, the multiplier is no longer required to satisfy the original constraint exactly, while the relaxed formulation becomes consistent with the original constraint when the relaxation residual approaches zero.

	The resulting
	relaxed formulation for $q(t)$ is given as follows:
	\begin{align}
		\begin{cases}
			\frac{dq(t)}{dt} = -\beta \left(\frac{dW}{dt} + q(t)\int_{0}^{1}
			\left(\frac{\partial_{\rho \rho} \kappa}{\left|\partial_{\rho}\mathbf{X}\right|^2} - \frac{\partial_{\rho} \kappa \,\partial_{\rho}\mathbf{X} \cdot \partial_{\rho \rho} \mathbf{X}}{\left|\partial_{\rho}\mathbf{X}\right|^4} + \frac{1}{2}\kappa^{3}\right)\partial_t \mathbf{X} \cdot \mathbf{n}\left|\partial_\rho \mathbf{X}\right|\,d\rho\right),\\
			q(0) = 1.
		\end{cases}
	\end{align}
	where $\beta>0$ denotes a relaxation parameter controlling the deviation from the original constraint. A smaller value of $\beta$ leads to a slower response of the multiplier
	$q(t)$ during the evolution. Consequently, the system \eqref{WM-adaptive-LM} admits
	the following equivalent reformulation:
	\begin{subequations}\label{WM-adaptive-RLM}
		\begin{align}
			&\partial_t \mathbf{X} = V\mathbf{n} + \left[\frac{P}{\mathcal{J}} \, \left( M_f \, \left|\partial_\rho \mathbf{X}\right| \right)^{-2} \, \partial_\rho \left( M_f \, \left|\partial_\rho \mathbf{X}\right| \right)\right]\boldsymbol{\tau} , \label{WM-adaptive-a}\\[4pt] 
			&V = q(t)\left(\frac{\partial_{\rho \rho} \kappa}{\left|\partial_{\rho}\mathbf{X}\right|^2} - \frac{\partial_{\rho} \kappa \,\partial_{\rho}\mathbf{X} \cdot \partial_{\rho \rho} \mathbf{X}}{\left|\partial_{\rho}\mathbf{X}\right|^4} + \frac{1}{2}\kappa^{3}\right), \label{WM-adaptive-b}\\[4pt]
			&\kappa = -\frac{\partial_{\rho \rho} \mathbf{X}\cdot \mat{n}}{\left|\partial_{\rho}\mathbf{X}\right|^2}, \\[4pt]
			& \frac{dq(t)}{dt} = -\beta \left(\frac{dW}{dt} + q(t)\int_{0}^{1}
			\left(\frac{\partial_{\rho \rho} \kappa}{\left|\partial_{\rho}\mathbf{X}\right|^2} - \frac{\partial_{\rho} \kappa \,\partial_{\rho}\mathbf{X} \cdot \partial_{\rho \rho} \mathbf{X}}{\left|\partial_{\rho}\mathbf{X}\right|^4} + \frac{1}{2}\kappa^{3}\right)\partial_t \mathbf{X} \cdot \mathbf{n}\left|\partial_\rho \mathbf{X}\right|\,d\rho\right)\label{WM-adaptive-d}.
		\end{align}    
	\end{subequations}

	\begin{thm}
		The reformulated system \eqref{WM-adaptive-RLM} obeys the following energy decay law: 
		\begin{align}
			\frac{dW_{RLM}}{dt} = -\int_0^1 V^2\left| \partial_{\rho} \mathbf{X}\right| d\rho\leq 0,
		\end{align}
		where 
		\begin{align} \label{W_RLM}
			W_{RLM}: = W+ \frac{q(t)-1}{\beta} =\frac{1}{2}\int_{\Gamma} \kappa ^2ds + \frac{q(t)-1}{\beta}.
		\end{align}
	\end{thm}
	\begin{proof}
		The variation of the Willmore energy gives
		\[
		\frac{dW}{dt}
		=
		-\int_0^1
		\left(
		\frac{\partial_{\rho \rho} \kappa}{|\partial_{\rho}\mathbf{X}|^2}
		-\frac{\partial_{\rho} \kappa \,\partial_{\rho}\mathbf{X}
			\cdot\partial_{\rho\rho}\mathbf X}{|\partial_{\rho}\mathbf X|^4}
		+\frac12\kappa^3
		\right)
		\partial_t\mathbf X\cdot\mathbf n
		|\partial_\rho\mathbf X|\,d\rho .
		\]
		Using the relaxation equation \eqref{WM-adaptive-d}, we obtain
		\[
		\begin{aligned}
			\frac{dW_{RLM}}{dt}
			=
			\frac{dW}{dt}+\frac1\beta\frac{dq}{dt}
			=
			-q(t)\int_0^1
			\left(
			\frac{\partial_{\rho \rho} \kappa}{|\partial_{\rho}\mathbf{X}|^2}
			-\frac{\partial_{\rho} \kappa \,\partial_{\rho}\mathbf{X}
				\cdot\partial_{\rho\rho}\mathbf X}{|\partial_{\rho}\mathbf X|^4}
			+\frac12\kappa^3
			\right)
			\partial_t\mathbf X\cdot\mathbf n
			|\partial_\rho\mathbf X|\,d\rho .
		\end{aligned}
		\]
		According to \eqref{WM-adaptive-b} together with
		$\partial_t\mathbf X\cdot\mathbf n=V,$
		we derive
		\[
		\frac{dW_{RLM}}{dt}
		=
		-\int_0^1
		V^2|\partial_\rho\mathbf X|\,d\rho
		\leq0,
		\]
		which ends the proof. 
	\end{proof}
	\begin{rem}
		At the continuous level, the relaxation variable satisfies
		$q(t)\equiv1$, and hence the additional term
		$(q(t)-1)/\alpha$ vanishes, recovering the original Willmore energy.
		At the discrete level, the relaxation variable $q^n$ may deviate from
		unity due to the temporal discretization. The asymptotic consistency
		and convergence of the relaxation variable have been established in
		the framework of relaxed Lagrange multiplier approaches; see
		\cite{Zhang2026RLMPF}. Therefore, the relaxed energy serves as a discrete energy
		functional to guarantee the energy stability of the fully discrete
		scheme.
	\end{rem}
\subsection{A-RLM-BDF1-FDM}

We first derive the first-order fully discrete A-RLM-BDF1-FDM. Taking $k=1$, i.e., $\alpha_0=1$ and $\alpha_1=-1$, in \eqref{eq:willmore_MM}, together with a backward Euler discretization of \eqref{WM-adaptive-d}, yields the following A-RLM-BDF1-FDM:
\begin{subequations}\label{eq:willmore_BDF1}
	\begin{align}
		&\frac{\mathbf{X}_i^{n+1} - \mathbf{X}_i^n}{\Delta t} = V_i^{n+1} \mathbf{n}_i^{n+1} + \left[\left(\frac{P_i^{n+1} M_{f,i}^{n+1}}{\mathcal{J}\left(M_{f,i}^{n+1} \left|\delta_\rho \mathbf{X}_i^{n+1}\right|\right)^2} \delta_{\rho\rho}\mathbf{X}_i^{n+1}\right) \cdot \boldsymbol{\tau}_i^{n+1} + \frac{P_i^{n+1} \delta_\rho M_{f,i}^{n+1} \left|\delta_\rho \mathbf{X}_i^{n+1}\right|}{\mathcal{J}\left(M_{f,i}^{n+1} \left|\delta_\rho \mathbf{X}_i^{n+1}\right|\right)^2}\right] \boldsymbol{\tau}_i^{n+1}, \label{eq:BDF1_a} \\[4pt]
		&V_i^{n+1} = q^{n+1} \left( \frac{\delta_{\rho\rho}\kappa_i^{n+1}}{\left|\delta_\rho \mathbf{X}_i^{n+1}\right|^2} - \frac{\delta_\rho \kappa_i^{n+1} \left(\delta_\rho \mathbf{X}_i^{n+1} \cdot \delta_{\rho\rho} \mathbf{X}_i^{n+1}\right)}{\left|\delta_\rho \mathbf{X}_i^{n+1}\right|^4} + \frac{1}{2}(\kappa_i^{n+1})^3 \right), \label{eq:BDF1_b}\\[4pt]
		&\kappa_i^{n+1} = -\frac{\delta_{\rho\rho}\mathbf{X}_i^{n+1} \cdot \mathbf{n}_i^{n+1}}{\left|\delta_\rho \mathbf{X}_i^{n+1}\right|^2}, \label{eq:BDF1_c}\\[4pt]
		&\frac{q^{n+1} - q^n}{\Delta t} = -\beta \frac{W^{n+1} - W^n}{\Delta t} \nonumber\\
		&\quad - \beta q^{n+1} \sum_i \left( \frac{\delta_{\rho\rho}\kappa_i^{n+1}}{\left|\delta_\rho \mathbf{X}_i^{n+1}\right|^2} - \frac{\delta_\rho \kappa_i^{n+1} \left(\delta_\rho \mathbf{X}_i^{n+1} \cdot \delta_{\rho\rho} \mathbf{X}_i^{n+1}\right)}{\left|\delta_\rho \mathbf{X}_i^{n+1}\right|^4} + \frac{1}{2}(\kappa_i^{n+1})^3 \right) \left(\frac{\mathbf{X}_i^{n+1} - \mathbf{X}_i^n}{\Delta t} \cdot \mathbf{n}_i^{n+1}\right) \left|\delta_\rho \mathbf{X}_i^{n+1}\right| \Delta\rho. \label{eq:BDF1_d}
	\end{align}
	\end{subequations}
		
For the subsequent energy analysis, we introduce the discrete Willmore energy
\begin{align}
	W^n := \frac{1}{2}\sum_i (\kappa_i^n)^2
	\left|\delta_\rho \mathbf{X}_i^n\right| \Delta\rho,
	\label{eq:discrete_willmore_energy}
\end{align}
and define the relaxed energy for the A-RLM-BDF1-FDM by
\begin{align}
	W_{\mathrm{RLM}}^{n} := W^n + \frac{q^n - 1}{\beta}.
	\label{eq:discrete_RLM_energy}
\end{align}
Then, we have the following energy stability result. 
	
	\begin{thm}[Discrete Energy Stability Law for A-RLM-BDF1-FDM]
The  A-RLM-BDF1-FDM \eqref{eq:willmore_BDF1} satisfies the discrete energy stability law:
		\begin{align}
			\frac{W_{\mathrm{RLM}}^{n+1} - W_{\mathrm{RLM}}^n}{\Delta t} = -\sum_i (V_i^{n+1})^2 \left|\delta_\rho \mathbf{X}_i^{n+1}\right| \Delta\rho \le 0, \label{eq:BDF1_dissipation}
		\end{align}
		which guarantees the unconditional discrete energy stability.
	\end{thm}
		
		\begin{proof}
			Taking the inner product of \eqref{eq:BDF1_a} with the unit normal vector $\mathbf{n}_i^{n+1}$ and utilizing the orthogonality property $\boldsymbol{\tau}_i^{n+1} \cdot \mathbf{n}_i^{n+1} = 0$, we extract the discrete normal velocity identity:
			\begin{align}
				\frac{\mathbf{X}_i^{n+1} - \mathbf{X}_i^n}{\Delta t} \cdot \mathbf{n}_i^{n+1} = V_i^{n+1}. \label{eq:BDF1_normal_vel}
			\end{align}
			Substituting \eqref{eq:BDF1_normal_vel} into \eqref{eq:BDF1_d} and recognizing from \eqref{eq:BDF1_b},
			the discrete relaxation equation reduces to
			\begin{align}
				\frac{q^{n+1} - q^n}{\Delta t} = -\beta \frac{W^{n+1} - W^n}{\Delta t} - \beta \sum_i (V_i^{n+1})^2 \left|\delta_\rho \mathbf{X}_i^{n+1}\right| \Delta\rho.
			\end{align}
			Dividing both sides by $\beta > 0$, rearranging the terms, and applying the definition of the relaxed discrete energy in \eqref{eq:discrete_RLM_energy}, we obtain
			\begin{align}
				\frac{W_{\mathrm{RLM}}^{n+1} - W_{\mathrm{RLM}}^n}{\Delta t} = \frac{W^{n+1} - W^n}{\Delta t} + \frac{1}{\beta}\frac{q^{n+1} - q^n}{\Delta t} = -\sum_i (V_i^{n+1})^2 \left|\delta_\rho \mathbf{X}_i^{n+1}\right| \Delta\rho,
			\end{align}
			which directly yields the discrete energy stability law.
			\end{proof}
To efficiently solve the A-RLM-BDF1-FDM, we employ a Picard-type iterative procedure at every time step. This approach updates the geometric variables through a sequence of linearized problems until convergence is achieved. Starting from the iteration $\left(\mathbf{X}_i^{\,n+1,m},\,V_i^{\,n+1,m},\,\kappa_i^{\,n+1,m}, q^{\,n+1,m}\right)$, the subsequent iteration $\left(\mathbf{X}_i^{\,n+1,m+1},\,V_i^{\,n+1,m+1},\,\kappa_i^{\,n+1,m+1}, q^{\,n+1,m+1}\right)$ is determined by solving
\begin{subequations}\label{eq:BDF1_iterative}
	\begin{align}
		&\frac{\mathbf{X}_i^{\,n+1,m+1} - \mathbf{X}_i^n}{\Delta t} = \left[\left(\frac{P_i^{n+1,m}\,M_{f,i}^{n+1,m}}{\mathcal{J}\left(M_{f,i}^{n+1,m}\,\left|\delta_\rho \mathbf{X}_i^{n+1,m}\right|\right)^{2}}
		\,\delta_{\rho\rho}\mathbf{X}_i^{n+1,m+1} 
		\right)\!\cdot\!\boldsymbol{\tau}_i^{n+1,m}\right]\boldsymbol{\tau}_i^{n+1,m} \nonumber \\
		&\hspace{3.2cm} + \left(\frac{P_i^{n+1,m}\,\delta_{\rho} M_{f,i}^{n+1,m}\,\left|\delta_\rho \mathbf{X}_i^{n+1,m}\right|}
		{\mathcal{J}\left(M_{f,i}^{n+1,m}\,\left|\delta_\rho \mathbf{X}_i^{n+1,m}\right|\right)^{2}}\right)\boldsymbol{\tau}_i^{n+1,m} + V_i^{\,n+1,m+1} \mathbf{n}_i^{\,n+1,m}, \label{eq:BDF1_iter_a} \\[4pt]
		&V_i^{\,n+1,m+1} = q^{\,n+1,m} \left(
		\frac{\delta_{\rho\rho}\kappa_i^{\,n+1,m+1}}{\left|\delta_\rho \mathbf{X}_i^{\,n+1,m}\right|^2}
		-\frac{\delta_\rho \kappa_i^{\,n+1,m}\,
			\left(\delta_\rho \mathbf{X}_i^{\,n+1,m}\!\cdot\!\delta_{\rho\rho}\mathbf{X}_i^{\,n+1,m+1}\right)}
		{\left|\delta_\rho \mathbf{X}_i^{\,n+1,m}\right|^4}
		+\tfrac{1}{2}\left(\kappa_i^{\,n+1,m}\right)^2\kappa_i^{\,n+1,m+1} \right), \label{eq:BDF1_iter_b} \\[4pt]
		&\kappa_i^{\,n+1,m+1} =
		-\frac{\delta_{\rho\rho}\mathbf{X}_i^{\,n+1,m+1}\!\cdot\!\mathbf{n}_i^{\,n+1,m}}
		{\left|\delta_\rho \mathbf{X}_i^{\,n+1,m}\right|^2}, \label{eq:BDF1_iter_c} \\[4pt]
		& \frac{q^{\,n+1,m+1} - q^n}{\Delta t} =-\beta \frac{W^{\,n+1,m+1} - W^n}{\Delta t} -\beta q^{\,n+1,m+1}
		\sum_i \left( \frac{\delta_{\rho\rho}\kappa_i^{\,n+1,m}}
		{\left|\delta_\rho\mathbf X_i^{\,n+1,m}\right|^2}
		- \frac{ \delta_\rho\kappa_i^{\,n+1,m}
			\left( \delta_\rho\mathbf X_i^{\,n+1,m}
			\cdot \delta_{\rho\rho}\mathbf X_i^{\,n+1,m}
			\right)}{\left|\delta_\rho\mathbf X_i^{\,n+1,m}\right|^4}\right. \nonumber\\
		&\quad
	\left.	+ \frac12(\kappa_i^{\,n+1,m})^3\right) \frac{\mathbf{X}_i^{\,n+1,m} - \mathbf{X}_i^n}{\Delta t} \cdot \mathbf{n}^{\,n+1,m}_i\left|\delta_\rho\mathbf X_i^{\,n+1,m}\right| \Delta\rho. \label{eq:BDF1_iter_d}
	\end{align}
\end{subequations}
The iteration continues until the following stopping criterion is satisfied:
\[
\max_i\!\left(
\left\|\mathbf{X}_i^{\,n+1,m+1}-\mathbf{X}_i^{\,n+1,m}\right\|
+\left|V_i^{\,n+1,m+1}-V_i^{\,n+1,m}\right|
+\left|\kappa_i^{\,n+1,m+1}-\kappa_i^{\,n+1,m}\right| +\left|q^{\,n+1,m+1}-q^{\,n+1,m}\right|\right)
<\varepsilon_{\mathrm{tol}}.
\]
Once this tolerance criterion is satisfied, we set 
\[
\left(\mathbf{X}_i^{\,n+1},\,V_i^{\,n+1},\,\kappa_i^{\,n+1},q^{\,n+1}\right)
=\left(\mathbf{X}_i^{\,n+1,m+1},\,V_i^{\,n+1,m+1},\,\kappa_i^{\,n+1,m+1},\,q^{\,n+1,m+1}\right).
\]		
		
\subsection{A-RLM-BDF2-FDM}

To achieve second-order temporal accuracy, we now consider the A-RLM-BDF2-FDM. Specifically, taking $k=2$ with $\alpha_0=\frac32$, $\alpha_1=-2$, and $\alpha_2=\frac12$ in \eqref{eq:willmore_MM}, together with a BDF2 discretization of \eqref{WM-adaptive-d}, we obtain the following A-RLM-BDF2-FDM:
\begin{subequations}\label{eq:willmore_BDF2}
	\begin{align}
		&\frac{3\mathbf{X}_i^{n+1} - 4\mathbf{X}_i^n + \mathbf{X}_i^{n-1}}{2\Delta t} = V_i^{n+1} \mathbf{n}_i^{n+1} + \left[\left(\frac{P_i^{n+1} M_{f,i}^{n+1}}{\mathcal{J}\left(M_{f,i}^{n+1} \left|\delta_\rho \mathbf{X}_i^{n+1}\right|\right)^2} \delta_{\rho\rho}\mathbf{X}_i^{n+1}\right) \cdot \boldsymbol{\tau}_i^{n+1} + \frac{P_i^{n+1} \delta_\rho M_{f,i}^{n+1} \left|\delta_\rho \mathbf{X}_i^{n+1}\right|}{\mathcal{J}\left(M_{f,i}^{n+1} \left|\delta_\rho \mathbf{X}_i^{n+1}\right|\right)^2}\right] \boldsymbol{\tau}_i^{n+1}, \label{eq:BDF2_a} \\[4pt]
		&V_i^{n+1} = q^{n+1} \left( \frac{\delta_{\rho\rho}\kappa_i^{n+1}}{\left|\delta_\rho \mathbf{X}_i^{n+1}\right|^2} - \frac{\delta_\rho \kappa_i^{n+1} \left(\delta_\rho \mathbf{X}_i^{n+1} \cdot \delta_{\rho\rho} \mathbf{X}_i^{n+1}\right)}{\left|\delta_\rho \mathbf{X}_i^{n+1}\right|^4} + \frac{1}{2}(\kappa_i^{n+1})^3 \right), \label{eq:BDF2_b}\\[4pt]
		&\kappa_i^{n+1} = -\frac{\delta_{\rho\rho}\mathbf{X}_i^{n+1} \cdot \mathbf{n}_i^{n+1}}{\left|\delta_\rho \mathbf{X}_i^{n+1}\right|^2}, \label{eq:BDF2_c}\\[4pt]
		&\frac{3q^{n+1} - 4q^n + q^{n-1}}{2\Delta t} = -\beta \frac{3W^{n+1} - 4W^n + W^{n-1}}{2\Delta t} \nonumber\\
		&\quad - \beta q^{n+1} \sum_i \left( \frac{\delta_{\rho\rho}\kappa_i^{n+1}}{\left|\delta_\rho \mathbf{X}_i^{n+1}\right|^2} - \frac{\delta_\rho \kappa_i^{n+1} \left(\delta_\rho \mathbf{X}_i^{n+1} \cdot \delta_{\rho\rho} \mathbf{X}_i^{n+1}\right)}{\left|\delta_\rho \mathbf{X}_i^{n+1}\right|^4} + \frac{1}{2}(\kappa_i^{n+1})^3 \right) \left(\frac{3\mathbf{X}_i^{n+1} - 4\mathbf{X}_i^n + \mathbf{X}_i^{n-1}}{2\Delta t} \cdot \mathbf{n}_i^{n+1}\right) \left|\delta_\rho \mathbf{X}_i^{n+1}\right| \Delta\rho. \label{eq:BDF2_d}
	\end{align}
\end{subequations}

For the A-RLM-BDF2-FDM, we continue to use the notation $W_{\rm RLM}^{n}$ for the relaxed energy, although its definition is modified as follows:
\begin{equation}\label{eq:RLM_energy_BDF2}
	W_{\rm RLM}^{n}
	:=
	\frac{3W^{n}-W^{n-1}}{2}+\frac{3q^{n}-q^{n-1}-2}{2\beta}.
\end{equation}

\begin{thm}[Discrete Energy Stability Law for BDF2]
The A-RLM-BDF2-FDM satisfies the discrete energy stability law:
	\begin{equation}\label{eq:RLM_energy_decay}
		\begin{aligned}
			\frac{W_{\rm RLM}^{n+1}-W_{\rm RLM}^{n}}{\Delta t}
			=-\sum_i(V_i^{n+1})^2|\delta_\rho\mathbf X_i^{n+1}|\Delta\rho
			\leq0,
		\end{aligned}
	\end{equation}
which implies that the A-RLM-BDF2-FDM is unconditionally energy-stable.
\end{thm}
\begin{proof}
	Taking the normal component of \eqref{eq:BDF2_a} and using the orthogonality between the tangential and normal directions, we obtain
	\begin{align*}
		\frac{3\mathbf X_i^{n+1}-4\mathbf X_i^n+\mathbf X_i^{n-1}}{2\Delta t}
		\cdot\mathbf n_i^{n+1}
		=
		V_i^{n+1}.
	\end{align*}
	Multiplying \eqref{eq:BDF2_d} by $1/\beta$ and combining it with \eqref{eq:BDF2_b}, we derive
	\begin{equation*}
		\frac{3W^{n+1}-4W^n+W^{n-1}}{2\Delta t}
		+
		\frac{3q^{n+1}-4q^n+q^{n-1}}{2\beta\Delta t}
		=
		-\sum_i(V_i^{n+1})^2|\delta_\rho\mathbf X_i^{n+1}|\Delta\rho,
	\end{equation*}
	which implies that 
	\begin{align*}
		\frac{W_{\rm RLM}^{n+1}-W_{\rm RLM}^{n}}{\Delta t}
		=\left(\frac{3W^{n+1}-W^{n}}{2\Delta t}+\frac{3q^{n+1}-q^{n}-2}{2\beta\Delta t}\right)-
		\left(\frac{3W^{n}-W^{n-1}}{2\Delta t}+\frac{3q^{n}-q^{n-1}-2}{2\beta\Delta t}\right)
		=
		-\sum_i(V_i^{n+1})^2|\delta_\rho\mathbf X_i^{n+1}|\Delta\rho\leq 0.
	\end{align*}
	Therefore, the A-RLM-BDF2-FDM satisfies the discrete energy stability law.
\end{proof}

To solve the nonlinear A-RLM-BDF2-FDM system efficiently, we apply the Picard-type iterative procedure at time step $n+1$. Starting from $\left(\mathbf{X}_i^{\,n+1,m},\,V_i^{\,n+1,m},\,\kappa_i^{\,n+1,m}, q^{\,n+1,m}\right)$, the next iterate $\left(\mathbf{X}_i^{\,n+1,m+1},\,V_i^{\,n+1,m+1},\,\kappa_i^{\,n+1,m+1}, q^{\,n+1,m+1}\right)$ is obtained by solving
\begin{subequations}\label{eq:BDF2_iterative}
	\begin{align}
		&\frac{3\mathbf{X}_i^{\,n+1,m+1} - 4\mathbf{X}_i^n + \mathbf{X}_i^{n-1}}{2\Delta t} = \left[\left(\frac{P_i^{n+1,m}\,M_{f,i}^{n+1,m}}{\mathcal{J}\left(M_{f,i}^{n+1,m}\,\left|\delta_\rho \mathbf{X}_i^{n+1,m}\right|\right)^{2}}
		\,\delta_{\rho\rho}\mathbf{X}_i^{n+1,m+1} 
		\right)\!\cdot\!\boldsymbol{\tau}_i^{n+1,m}\right]\boldsymbol{\tau}_i^{n+1,m} \nonumber \\
		&\hspace{4.2cm} + \left(\frac{P_i^{n+1,m}\,\delta_{\rho} M_{f,i}^{n+1,m}\,\left|\delta_\rho \mathbf{X}_i^{n+1,m}\right|}
		{\mathcal{J}\left(M_{f,i}^{n+1,m}\,\left|\delta_\rho \mathbf{X}_i^{n+1,m}\right|\right)^{2}}\right)\boldsymbol{\tau}_i^{n+1,m} + V_i^{\,n+1,m+1} \mathbf{n}_i^{\,n+1,m}, \label{eq:BDF2_iter_a} \\[4pt]
		&V_i^{\,n+1,m+1} = q^{\,n+1,m} \left(
		\frac{\delta_{\rho\rho}\kappa_i^{\,n+1,m+1}}{\left|\delta_\rho \mathbf{X}_i^{\,n+1,m}\right|^2}
		-\frac{\delta_\rho \kappa_i^{\,n+1,m}\,
			\left(\delta_\rho \mathbf{X}_i^{\,n+1,m}\!\cdot\!\delta_{\rho\rho}\mathbf{X}_i^{\,n+1,m+1}\right)}
		{\left|\delta_\rho \mathbf{X}_i^{\,n+1,m}\right|^4}
		+\tfrac{1}{2}\left(\kappa_i^{\,n+1,m}\right)^2\kappa_i^{\,n+1,m+1} \right), \label{eq:BDF2_iter_b} \\[4pt]
		&\kappa_i^{\,n+1,m+1} =
		-\frac{\delta_{\rho\rho}\mathbf{X}_i^{\,n+1,m+1}\!\cdot\!\mathbf{n}_i^{\,n+1,m}}
		{\left|\delta_\rho \mathbf{X}_i^{\,n+1,m}\right|^2}, \label{eq:BDF2_iter_c} \\[4pt]
		& \frac{3q^{\,n+1,m+1} - 4q^n + q^{n-1}}{2\Delta t} =-\beta \frac{3W^{\,n+1,m+1} - 4W^n + W^{n-1}}{2\Delta t} -\beta q^{\,n+1,m+1}
		\sum_i\left( \frac{\delta_{\rho\rho}\kappa_i^{\,n+1,m}}
		{\left|\delta_\rho\mathbf X_i^{\,n+1,m}\right|^2}- \frac{ \delta_\rho\kappa_i^{\,n+1,m}
			\left( \delta_\rho\mathbf X_i^{\,n+1,m}
			\cdot \delta_{\rho\rho}\mathbf X_i^{\,n+1,m}
			\right)}{\left|\delta_\rho\mathbf X_i^{\,n+1,m}\right|^4}\right.\nonumber\\
		&\quad  
		\left.
		+ \frac12(\kappa_i^{\,n+1,m})^3\right)
		 \left(\frac{3\mathbf{X}_i^{\,n+1,m} - 4\mathbf{X}_i^n + \mathbf{X}_i^{n-1}}{2\Delta t} \cdot \mathbf{n}^{\,n+1,m}_i\right)\left|\delta_\rho\mathbf X_i^{\,n+1,m}\right| \Delta\rho. \label{eq:BDF2_iter_d}
	\end{align}
\end{subequations}
The iteration terminates when the error falls below $\varepsilon_{\mathrm{tol}}$:
\[
\max_i\!\left(
\left\|\mathbf{X}_i^{\,n+1,m+1}-\mathbf{X}_i^{\,n+1,m}\right\|
+\left|V_i^{\,n+1,m+1}-V_i^{\,n+1,m}\right|
+\left|\kappa_i^{\,n+1,m+1}-\kappa_i^{\,n+1,m}\right| +\left|q^{\,n+1,m+1}-q^{\,n+1,m}\right|\right)
<\varepsilon_{\mathrm{tol}}.
\]

	\begin{rem}
		Besides incorporating the mesh adaptation mechanism directly into the
		geometric evolution through an appropriate tangential velocity, another
		possible strategy is to perform mesh redistribution or reparametrization
		after each time step. Although such post-processing mesh adjustment
		techniques have been widely adopted in geometric evolution problems
		\cite{Dziuk2002,DeckelnickDziuk2009}, they are generally regarded as
		auxiliary procedures for improving mesh quality rather than intrinsic
		components of the evolution system. For completeness, the A-WAR algorithm based on adaptive monitor
		functions is provided in Appendix B. This strategy extends the standard
		arc-length redistribution procedure by incorporating local geometric
		information, including curvature and curvature variation, into the
		monitor function, together with an adaptive selection mechanism that
		determines suitable monitor functions according to the geometric
		features of the evolving curve. This auxiliary procedure provides a
		practical approach for improving mesh quality.
	\end{rem}

\section{Numerical experiments}\label{sec5}
In this section, we present numerical experiments for the A-BDF$k$-FDMs, the A-RLM-BDF$k$-FDMs ($k=1,2$) and the A-WAR algorithm. These experiments are conducted to investigate their convergence properties and demonstrate their effectiveness in practical computations.

For the implementation of our numerical method, an initial configuration 
\((\mathbf{X}^{0}, V^{0}, \kappa^{0})\) is required. or the A-RLM-BDF$k$-FDMs, the relaxation variable $q$ is additionally introduced, and its initial value is set as $q^{0}=1$ according to its equilibrium state. To obtain a compatible starting
value for the normal velocity \(V^{0}\) and the mean curvature \(\kappa^{0}\),
we begin with a smooth initial curve \(\Gamma_{0}\), for which both the curvature 
\(\kappa^{0} = -\mathbf{n}\cdot \partial_{ss} \mathbf{X}^{0}\) and the velocity 
\(V^{0} = \partial_{ss}\kappa^{0} + \tfrac{1}{2} (\kappa^{0})^{3}\) can be computed 
explicitly. These quantities are then adopted as the initial approximation 
\((\mathbf{X}^{0},V^{0}, \kappa^{0})\) in our iterative solver. 
The stopping criterion for the iteration is set to \(\mathrm{tol} = 10^{-8}\), and we select the relax time constant $\mathcal J=0.5$.

\begin{example}[Convergence tests]
In this example, we assess the accuracy and convergence rate of the A-BDF$k$-FDMs, the A-RLM-BDF$k$-FDMs and the
A-WAR algorithm.
The initial curve is chosen as the unit circle,
\[
\mathbf{X}^0(\rho_j)
    = \left(\cos(2\pi \rho_j),\, \sin(2\pi \rho_j)\right)^{\mathsf T},
    \qquad 0 \le j \le M-1.
\]
Following \cite{barrett2008parametric}, the exact solution to the Willmore flow~\eqref{WM} is given by
\[
\mathbf{X}(\rho,t) = R(t)\,\mathbf{X}^0(\rho), \qquad
V(\rho,t) = \tfrac12 R(t)^{-3}, \qquad
\kappa(\rho,t) = R(t)^{-1},
\qquad
q(t) = 1,
\]
for \(\rho \in \mathbb{T}\) and \(t \ge 0\), where \(R(t) = (1+2t)^{1/4}\).
We introduce the combined geometric variable \(\mathbf{U} := (\mathbf{X}, V, \kappa)\) together with its numerical approximation at \(t = T\), defined by
\[
\mathbf{U}_{h_\ell}^{N_\ell}
    := \left(\mathbf{X}_{h_\ell}^{N_\ell},\, V_{h_\ell}^{N_\ell},\, \kappa_{h_\ell}^{N_\ell}\right),
    \qquad
    N_\ell = \frac{T}{\Delta t_\ell}.
\]
The error and the associated experimental convergence order are defined by
\[
  e^h_\ell = \left\| \mathbf{U}_{h_\ell}^{N_\ell} - \mathbf{U}(\cdot,T) \right\|_\infty,
  \qquad 
  order_\ell =
    \frac{\log\!\left(e^h_\ell / e^h_{\ell+1}\right)}
         {\log\!\left(\Delta t_\ell / \Delta t_{\ell+1}\right)}.
\]
To facilitate the simultaneous measurement of temporal and spatial convergence rates, we impose the relation \(h^{2} \approx \Delta t^{\,k}\).
Accordingly, the pair \((h_\ell, \Delta t_\ell)\) is updated as
\begin{equation*}
\left(h_{\ell+1}, \Delta t_{\ell+1}\right) =
\begin{cases}
\left(h_\ell/2,\ \Delta t_\ell/4\right), & \text{BDF1},\\[3mm]
\left(h_\ell/2,\ \Delta t_\ell/2\right), & \text{BDF2},\\[3mm]
\left(h_\ell/2^{3/2},\ \Delta t_\ell/2\right), & \text{BDF3},\\[3mm]
\left(h_\ell/4,\ \Delta t_\ell/2\right), & \text{BDF4}.
\end{cases}
\end{equation*}
Figs.~\ref{fig:ABDFktimeorder}--\ref{fig:AWARtimeorder} present the errors obtained from the A-BDF$k$-FDMs, the A-RLM-BDF$k$-FDMs and the A-WAR algorithm. In both settings, the numerical results exhibit convergence rates that agree remarkably well with the theoretical predictions $O(\Delta t^{\,k})$, thereby confirming first- through fourth-order temporal accuracy for BDF1--BDF4. For the A-RLM-BDF$k$-FDMs, the
computed errors for BDF1 and BDF2 also achieve the expected temporal
convergence rates of first and second order, respectively. Moreover, due to the prescribed coupling between the spatial and temporal step sizes, the computations also demonstrate a second-order convergence rate in space, which is fully consistent with the expected accuracy $O(h^{2})$.

\begin{figure}[!htbp]
    \centering
    \includegraphics[width=0.85\linewidth]{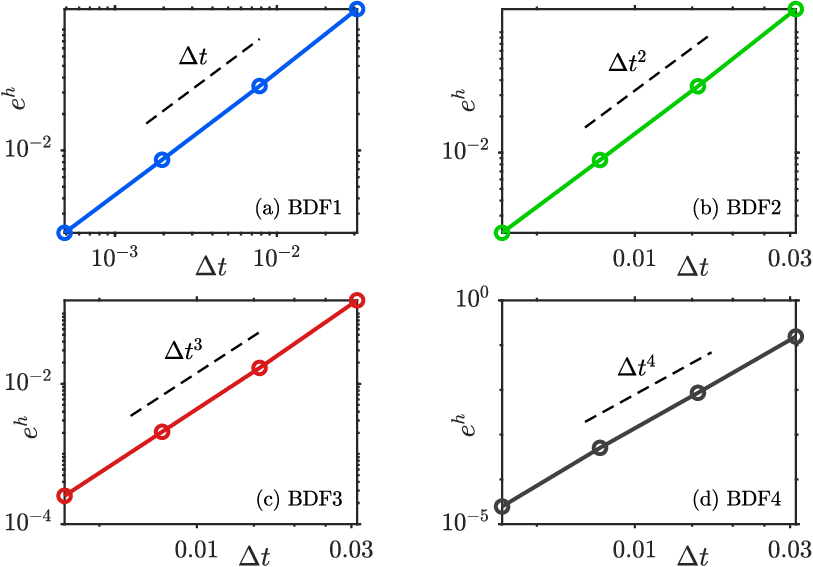}
    \caption{Plot of temporal errors $e^{h}$ for the A-BDF$k$-FDM using BDF$k$ time discretizations at $T=2$:  
(a) BDF1, (b) BDF2, (c) BDF3, and (d) BDF4.}
    \label{fig:ABDFktimeorder}
\end{figure}
\begin{figure}
	\centering
	\includegraphics[width=0.42\linewidth]{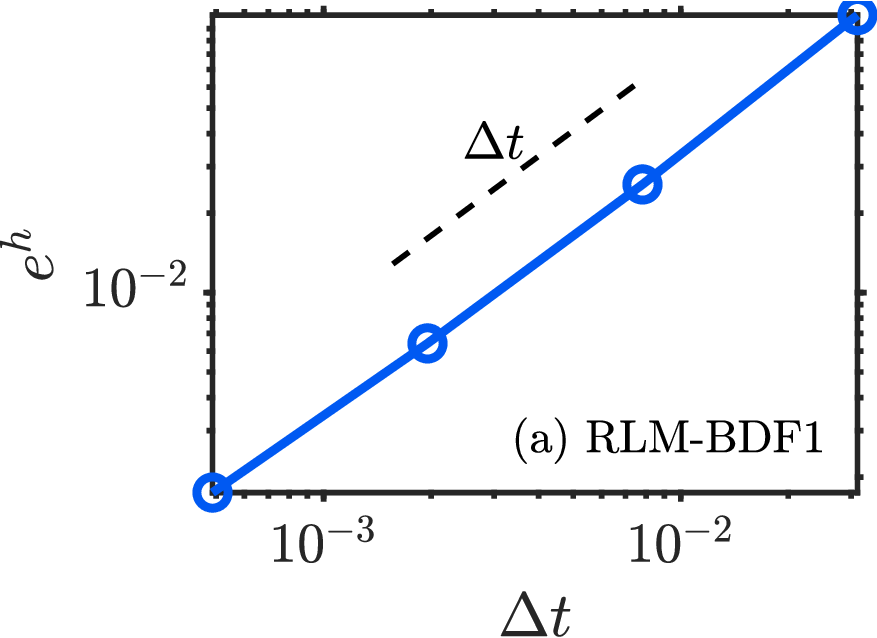}
	\hspace{0.04\textwidth}
	\includegraphics[width=0.42\linewidth]{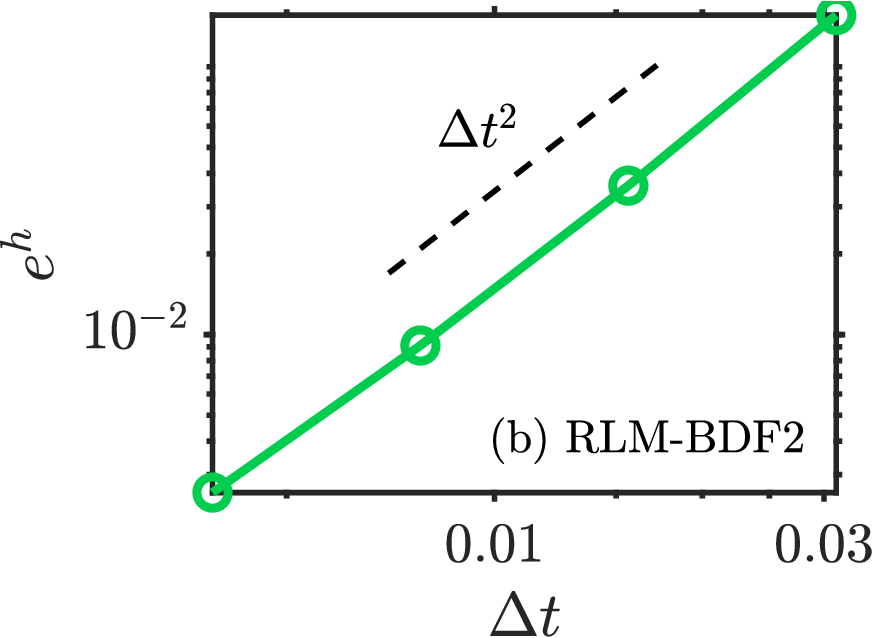}
	\caption{Plot of temporal errors $e^{h}$ for the A-RLM-BDF$k$-FDM using BDF$k$ time discretizations at $T=2$:  
		(a) BDF1, (b) BDF2.}
	\label{fig:ARLMBDFktimeorder}
\end{figure}
\begin{figure}[!htbp]
	\centering
	\includegraphics[width=0.85\linewidth]{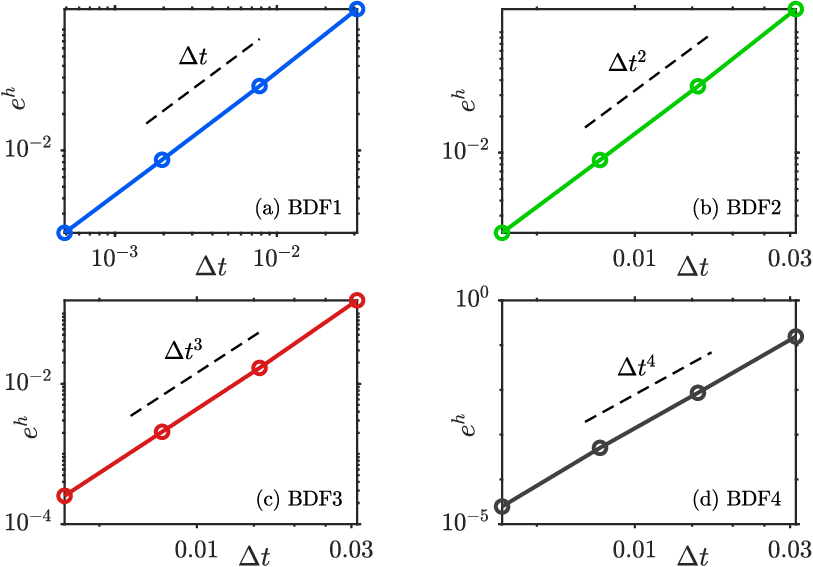}
	\caption{Plot of temporal errors $e^{h}$ for the A-WAR algorithm using BDF$k$ time discretizations at $T=2$:  
		(a) BDF1, (b) BDF2, (c) BDF3, and (d) BDF4.}
	\label{fig:AWARtimeorder}
\end{figure}

\end{example}
\begin{example}
	In this example, we compare the mesh distribution and mesh quality of three numerical methods, namely the BDF$k$-FDMs, the A-WAR algorithm, and the A-BDF$k$-FDMs, during the curve evolution process. For convenience, we restrict our attention to the case \(k=1\). We consider three types of initial curves and examine the corresponding grid point distributions generated by the three methods. 
	It can be observed from Fig. \ref{fig:mesh1}, Fig. \ref{fig:mesh2} and Fig. \ref{fig:mesh3} that, as the geometric complexity of the initial curves increases, both adaptive methods proposed in this work are able to maintain good mesh quality throughout the evolution. This, in turn, ensures the reliability and correctness of the computed curve evolution.

	In addition, to evaluate the mesh quality, several quantitative measures are employed. 
	In particular, the ratio between the maximum and minimum arc lengths,
	\[
	R_1(\Delta s):=\frac{\max(\Delta s)}{\min(\Delta s)},
	\]
	is used as a standard indicator of mesh uniformity. 
	In addition, the mesh quality measure
	\[
	R_2(M_f, \Delta s):=\frac{\max\left(M_f \Delta s\right)}{\min\left(M_f \Delta s\right)}
	\]
	is adopted to assess the consistency of the mesh in the equidistribution sense induced by the monitor function \(M_f\).
	As shown in Fig.~\ref{fig:mesh1x}, Fig.~\ref{fig:mesh2x} and Fig.~\ref{fig:mesh3x}, the two adaptive methods proposed in this work, namely the A-WAR algorithm and the A-BDF$k$-FDMs, are able to rapidly form a near-equidistributed mesh at the early stage of the evolution. 
	For both adaptive methods, the two mesh quality indicators remain close to unity throughout the entire computation, indicating good mesh uniformity and a strong adherence to the equidistribution principle induced by the monitor function.

	In contrast, the BDF$k$-FDMs exhibit markedly different mesh evolution behaviors. 
	As shown in Fig.~\ref{fig:mesh1x}, the mesh quality indicators of the BDF$k$-FDMs remain bounded during the evolution and gradually converge to finite constants, indicating that, for certain initial curves, the mesh maintains a certain level of overall stability. 
	However, these limiting values are clearly larger than unity, implying that the resulting mesh does not attain an ideal mesh distribution, and its uniformity and adaptivity are significantly inferior to those of the adaptive methods. 
	In Fig.~\ref{fig:mesh2x}, both mesh quality indicators deteriorate rapidly within a short time, indicating that mesh degradation severely impairs the long-time numerical computation. 
	More critically, in Fig.~\ref{fig:mesh3x}, both mesh quality indicators experience a sharp growth during the intermediate stage of the evolution and later decrease to relatively small values; however, as can be observed from Fig.~\ref{fig:mesh3}, the corresponding evolution curve has already become completely incorrect.
	
	These results indicate that, for initial curves with high curvature or pronounced curvature variations, the BDF$k$-FDMs are prone to severe mesh distortion and strong point clustering during the evolution, which significantly undermines the reliability of the numerical solution. 
	In some cases, such mesh degradation may even lead to a complete failure of the numerical computation. 
	In contrast, the A-WAR algorithm and the A-BDF$k$-FDMs proposed in this work are able to effectively control the mesh distribution and maintain good mesh quality throughout the evolution, thereby significantly enhancing the stability and reliability of the numerical simulations.

	
	\begin{figure}[!htbp]
		\centering
		\includegraphics[width=0.25\linewidth]{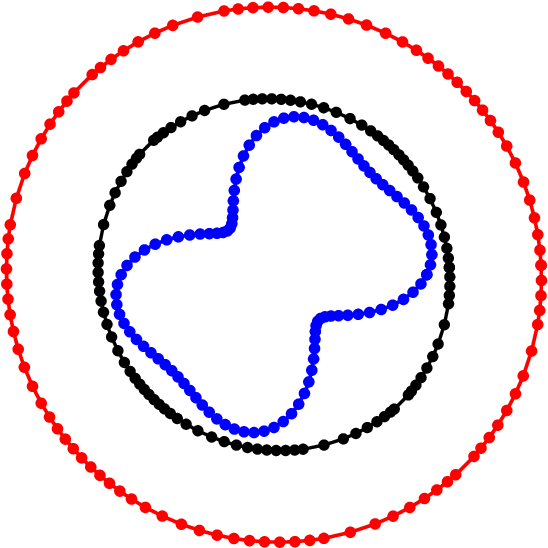}
		\hspace{0.05\textwidth}
		\includegraphics[width=0.25\linewidth]{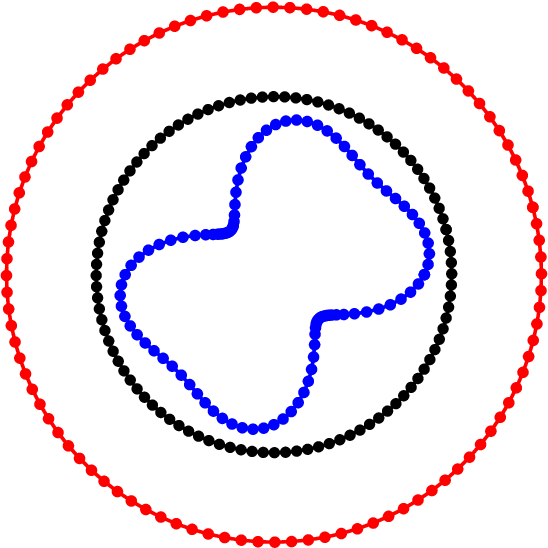}
		\hspace{0.05\textwidth}
		\includegraphics[width=0.25\linewidth]{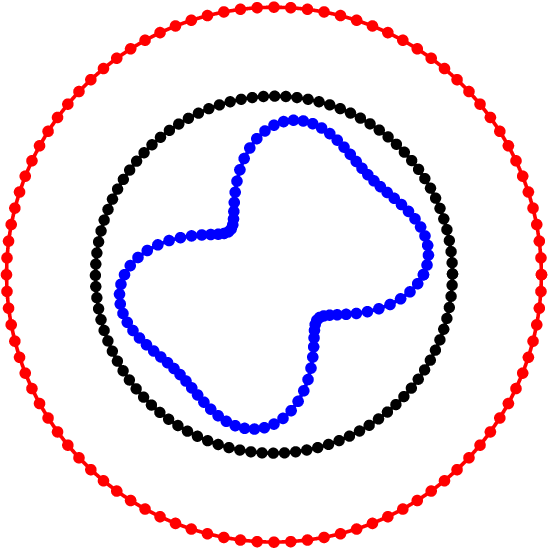}
		\caption{
			Mesh evolution computed by different numerical methods for the initial curve defined by
			$x =\cos(2\pi\rho) \left[1 + 0.3\sin(4\pi\rho) + 0.2\cos(8\pi\rho)\right], y = \sin(2\pi\rho)\left[1 + 0.3\sin(4\pi\rho) + 0.2\cos(8\pi\rho)\right]$.
			From left to right, the results correspond to the BDF$k$-FDMs, the A-WAR algorithm,
			and the A-BDF$k$-FDMs. The blue, black, and red curves represent the configurations
			at \( t = 0 \), \( t = 1.5 \), and \( t = 10 \), respectively.
		}
		\label{fig:mesh1}
	\end{figure}

	\begin{figure}[!htbp]
		\centering
		\includegraphics[width=0.4\linewidth]{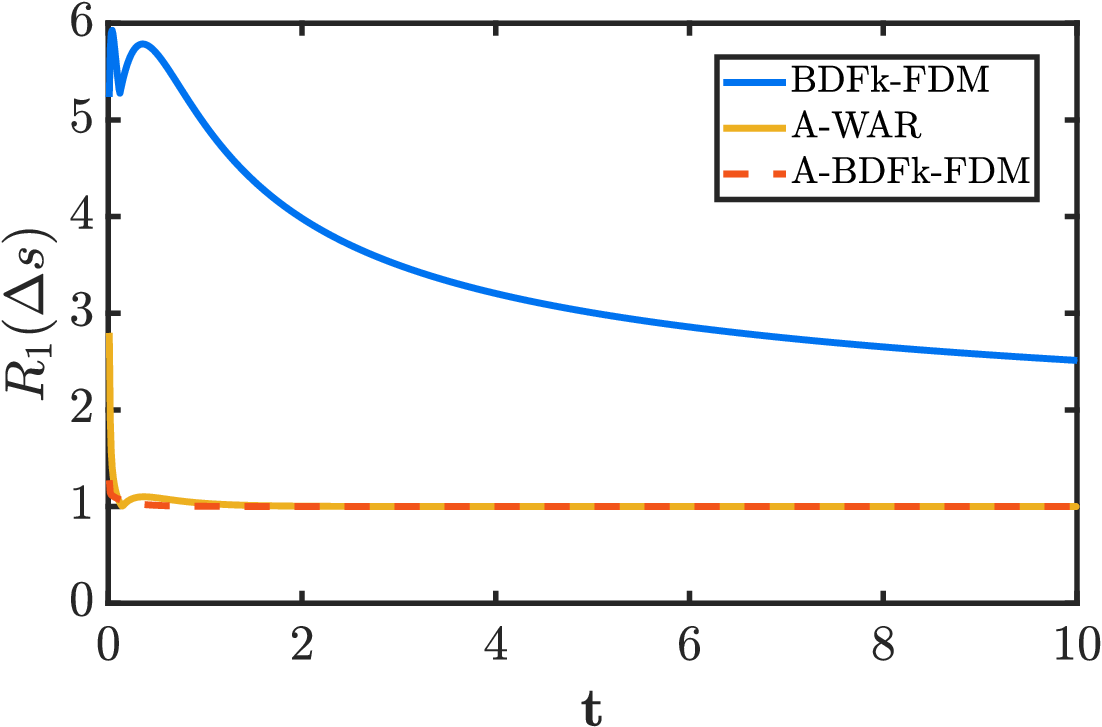}
		\hspace{0.05\textwidth}
		\includegraphics[width=0.4\linewidth]{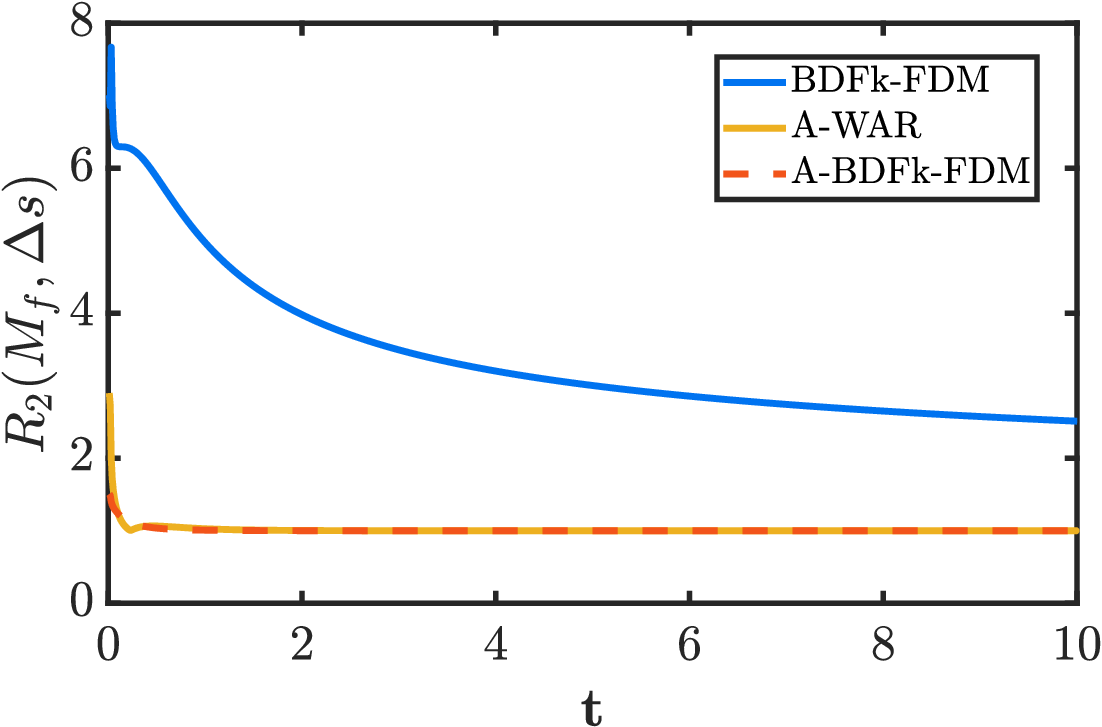}
		\caption{Time evolution of the mesh quality indicators
			$R_1(\Delta s)$ (left) and $R_2(M_f,\Delta s)$ (right)
			for the initial curve shown in Fig~\ref{fig:mesh1}.
			The results correspond to the BDF$k$-FDMs, the A-WAR algorithm,
			and the A-BDF$k$-FDMs.}
		\label{fig:mesh1x}
	\end{figure}

	\begin{figure}[!htbp]
		\centering
		\includegraphics[width=0.25\linewidth]{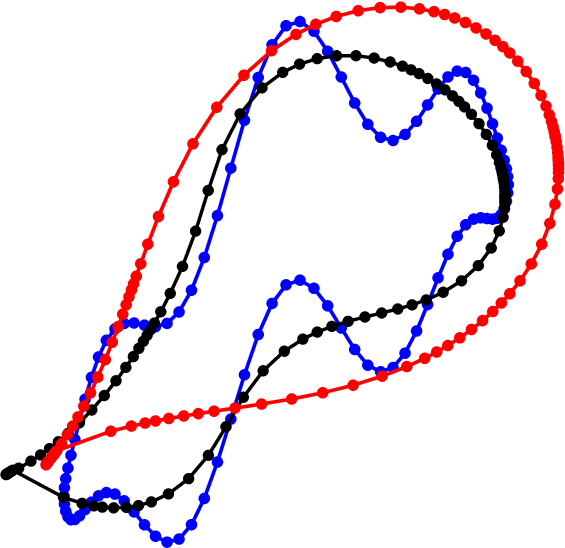}
		\hspace{0.05\textwidth}
		\includegraphics[width=0.25\linewidth]{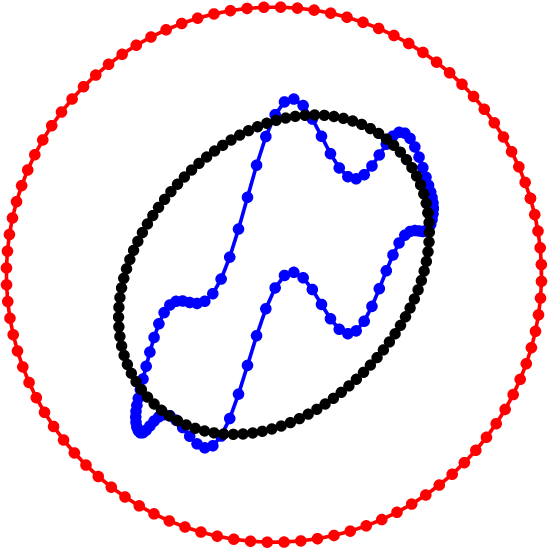}
		\hspace{0.05\textwidth}
		\includegraphics[width=0.25\linewidth]{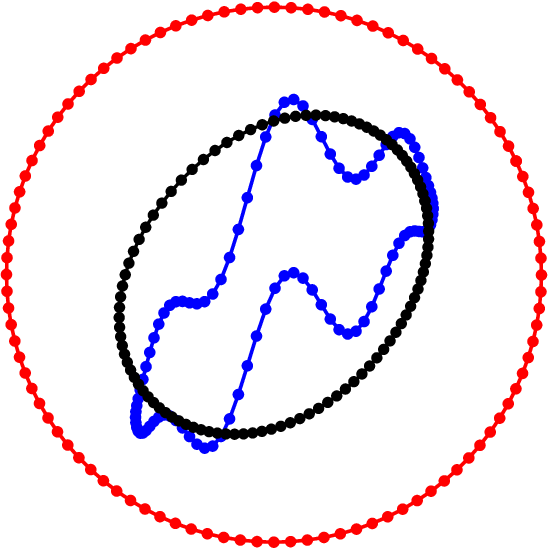}
		\caption{Mesh evolution computed by different numerical methods for the initial curve defined by
		$x = 1.2 \cos\left( 2\pi \rho \right), y = 0.5 \sin\left( 2\pi \rho \right)+ \sin\left( \cos\left( 2\pi \rho \right) \right)
			+ \sin\left( 2\pi \rho \right)\left[0.2 + \sin\left( 2\pi \rho \right)\sin^2\left( 6\pi\rho \right)\right]$.
			From left to right, the results correspond to the BDF$k$-FDMs, the A-WAR algorithm,
			and the A-BDF$k$-FDMs. The blue, black, and red curves represent the configurations
			at \( t = 0 \), \( t = 4 \), and \( t = 10 \), respectively.}
		\label{fig:mesh2}
	\end{figure}

	\begin{figure}[!htbp]
		\centering
		\includegraphics[width=0.4\linewidth]{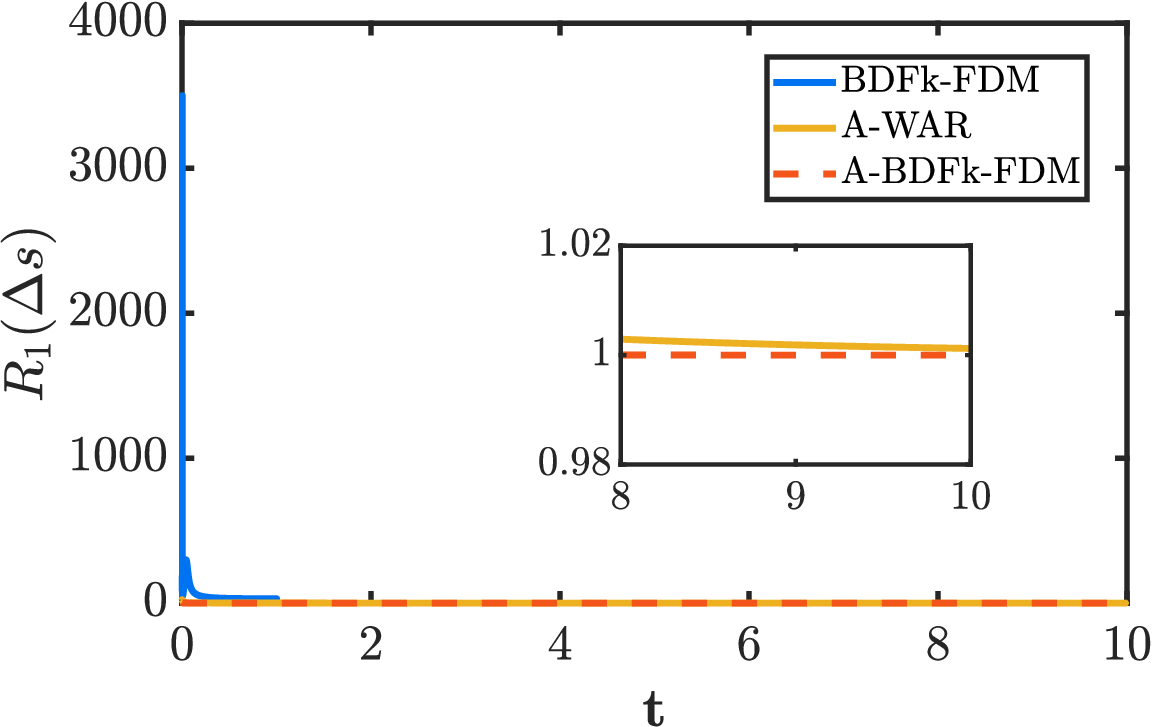}
		\hspace{0.05\textwidth}
		\includegraphics[width=0.4\linewidth]{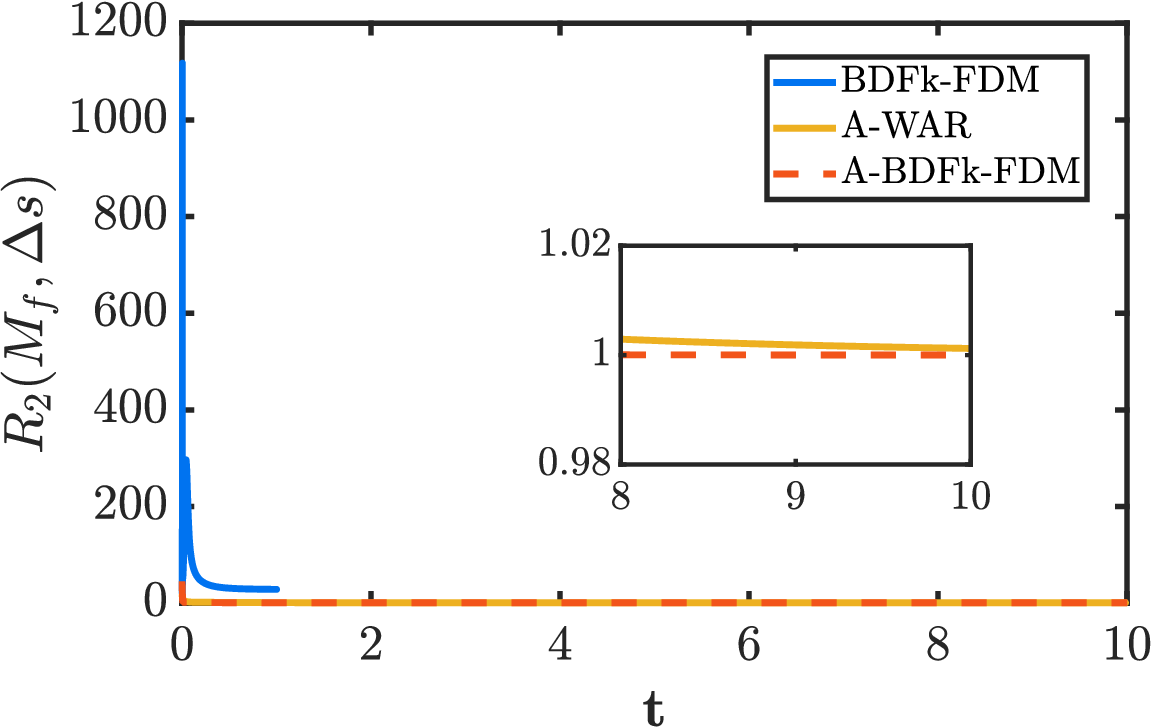}
		\caption{Time evolution of the mesh quality indicators
			$R_1(\Delta s)$ (left) and $R_2(M_f,\Delta s)$ (right)
			for the initial curve shown in Fig~\ref{fig:mesh2}.
			The results correspond to the BDF$k$-FDMs, the A-WAR algorithm,
			and the A-BDF$k$-FDMs.}
		\label{fig:mesh2x}
	\end{figure}

	\begin{figure}[!htbp]
		\centering
		\raisebox{5mm}{\includegraphics[width=0.32\linewidth]{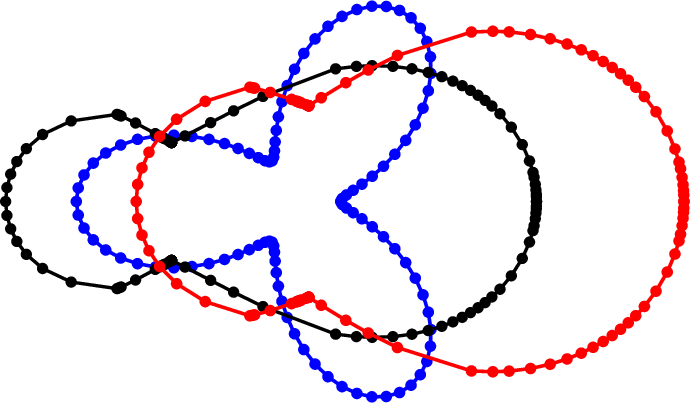}}
		\hspace{0.02\textwidth}
		\includegraphics[width=0.25\linewidth]{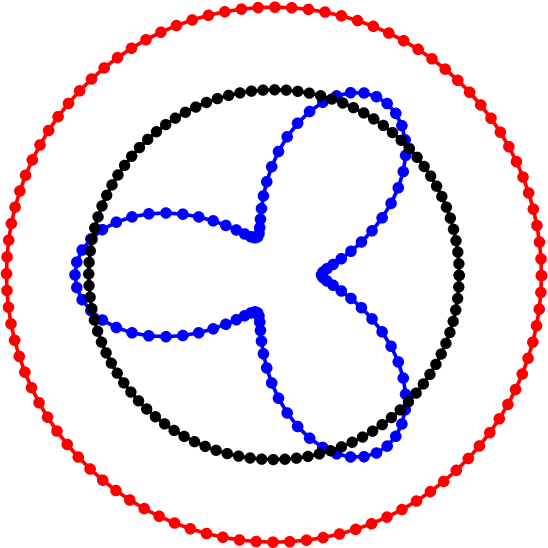}
		\hspace{0.05\textwidth}
		\includegraphics[width=0.25\linewidth]{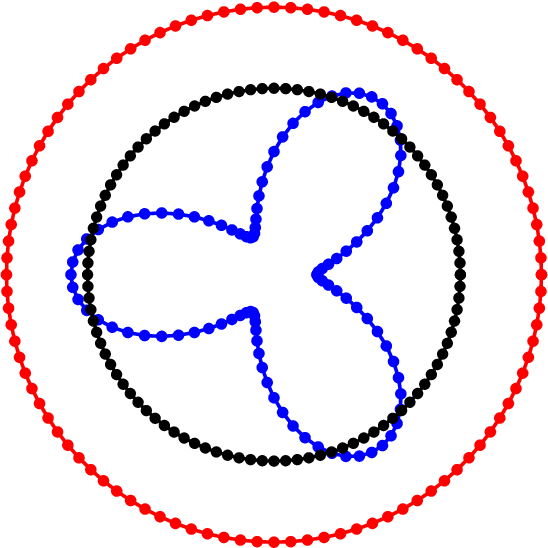}
		\caption{Mesh evolution computed by different numerical methods for the initial curve defined by
			$x= \cos(2\pi\rho)\left[ 1 - 0.65 \cos(6\pi\rho) \right] , 
			y= \sin(2\pi\rho) \left[1 - 0.65 \cos(6\pi\rho) \right] $.
			From left to right, the results correspond to the BDF$k$-FDMs, the A-WAR algorithm,
			and the A-BDF$k$-FDMs. The blue, black, and red curves represent the configurations
			at \( t = 0 \), \( t = 1.5 \), and \( t = 10 \), respectively.}
		\label{fig:mesh3}
	\end{figure}
	
	\begin{figure}[!htbp]
		\centering
		\includegraphics[width=0.4\linewidth]{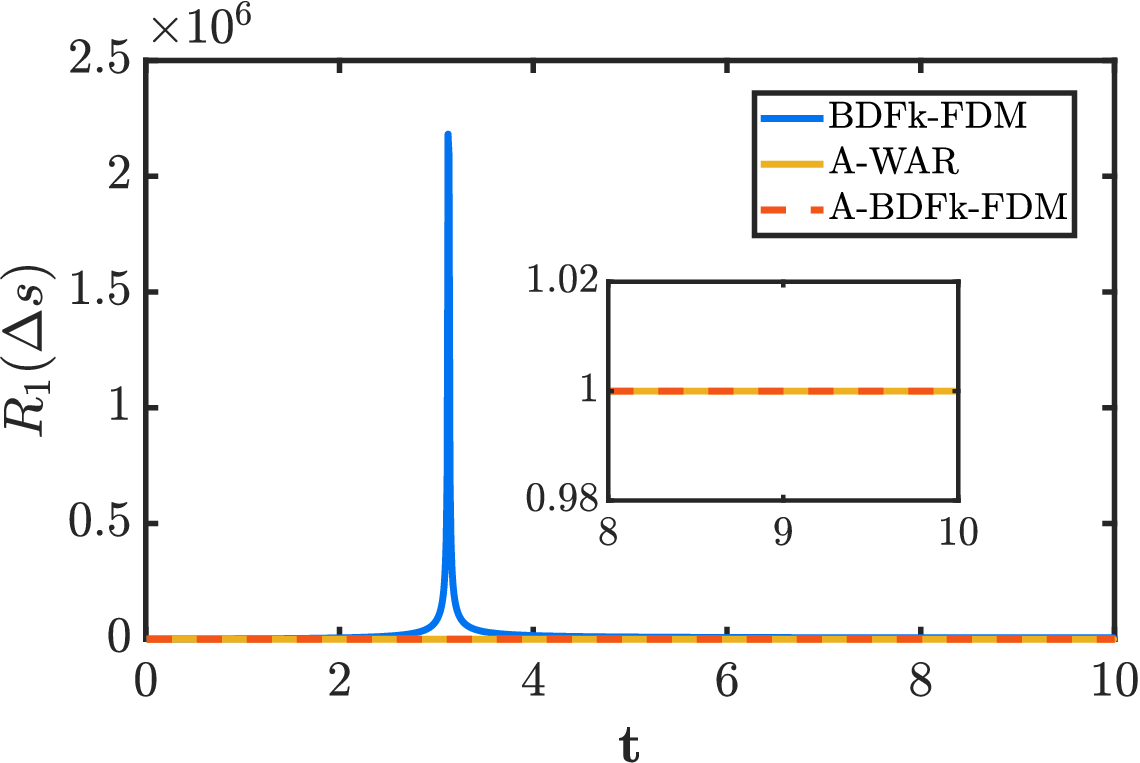}
		\hspace{0.05\textwidth}
		\includegraphics[width=0.4\linewidth]{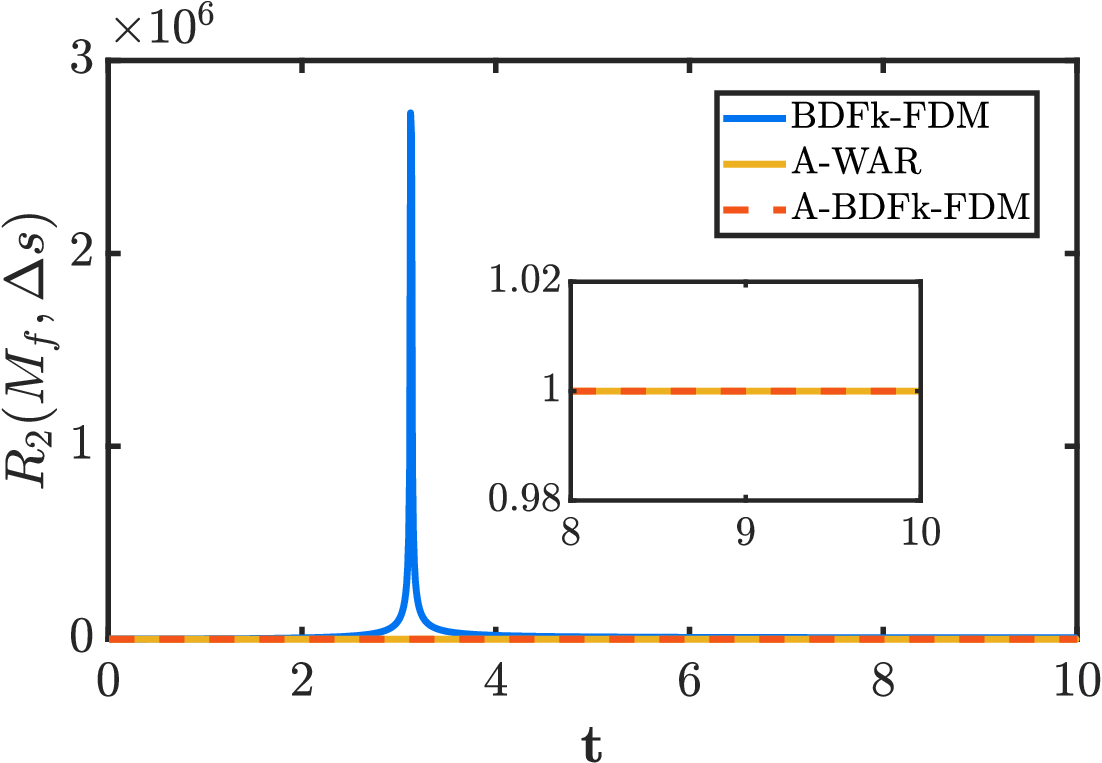}
		\caption{Time evolution of the mesh quality indicators
			$R_1(\Delta s)$ (left) and $R_2(M_f,\Delta s)$ (right)
			for the initial curve shown in Fig~\ref{fig:mesh3}.
			The results correspond to the BDF$k$-FDMs, the A-WAR algorithm,
			and the A-BDF$k$-FDMs.}
		\label{fig:mesh3x}
	\end{figure}
\end{example}
\begin{example} 
	(Energy stability) To verify the unconditional energy stability property of the proposed A-RLM-BDF$k$-FDMs and to investigate the influence of the relaxation parameter $\beta$, we consider the evolution of the Willmore energy for different values of $\beta$. The initial curve is chosen as
	\[
	\mathbf{X}(\rho,0) = \left( 1.5 \cos(2\pi \rho_j),\, \sin(2\pi \rho_j)\right)^{\mathsf T}, \qquad 0 \le j \le M-1.
	\]
	The time step size is set as $\Delta t = 0.01$, and the final simulation time is $T = 2$. Three different relaxation parameters, namely $\beta = 0.1$, $\beta = 0.01$, and $\beta = 0.001$, are considered. For comparison, the accurate energy $W_{\mathrm{accurate}}(t)$ is computed using the A-RLM-BDF$k$-FDMs with a refined time step size $\Delta t = 10^{-5}$ and $\beta = 10^{-5}$.
	
	The numerical results computed by the A-RLM-BDF1-FDM and A-RLM-BDF2-FDM are depicted in Fig.~\ref{fig:RLM_energy_bdf1} and Fig.~\ref{fig:RLM_energy_bdf2}, respectively, where $W_{\mathrm{RLM}}(t)$ denotes the modified energy and $W(t)$ represents the original Willmore energy evaluated at the numerical solution. As observed from the left panels of both figures, A-RLM-BDF$k$-FDMs strictly preserve the unconditional energy stability property throughout the simulation. Moreover, as the relaxation parameter $\beta$ decreases, the modified energy curves rapidly converge to the accurate energy $W_{\mathrm{accurate}}(t)$. To further examine the impact of the relaxation term, the middle panels demonstrate the energy difference $|W_{\mathrm{RLM}} - W|$. It is clear that reducing $\beta$ from $0.1$ to $0.001$ leads to a substantial decrease in the deviation between the modified and original energies by several orders of magnitude, confirming that the auxiliary term becomes virtually negligible for small values of $\beta$. Finally, the right panels illustrate the absolute error $|W_{\mathrm{RLM}} - W_{\mathrm{accurate}}|$, where smaller values of $\beta$ consistently yield superior numerical accuracy. 
	
	Through extensive numerical investigations, we observe that the nonlinear iterative solvers for other relaxation approaches fail to converge when combined with the proposed adaptive methods. By comparison, the proposed A-RLM-BDF$k$-FDMs exhibit robust convergence and effectively preserve the energy stability property. Furthermore, decreasing the relaxation parameter $\beta$ enhances the accuracy of the modified energy, making it increasingly consistent with the original Willmore energy.
	\begin{figure}
		\centering
		\includegraphics[width=0.3\linewidth]{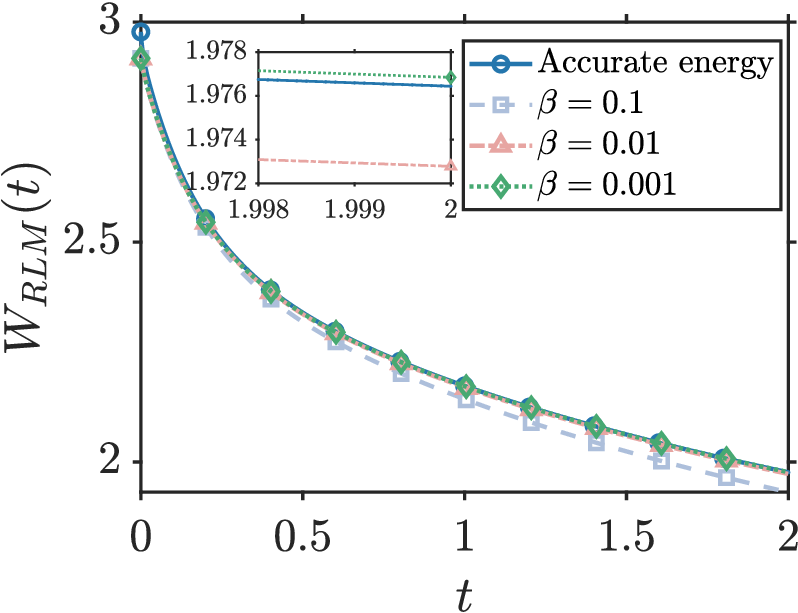}
		\hspace{0.02\textwidth}
		\includegraphics[width=0.3\linewidth]{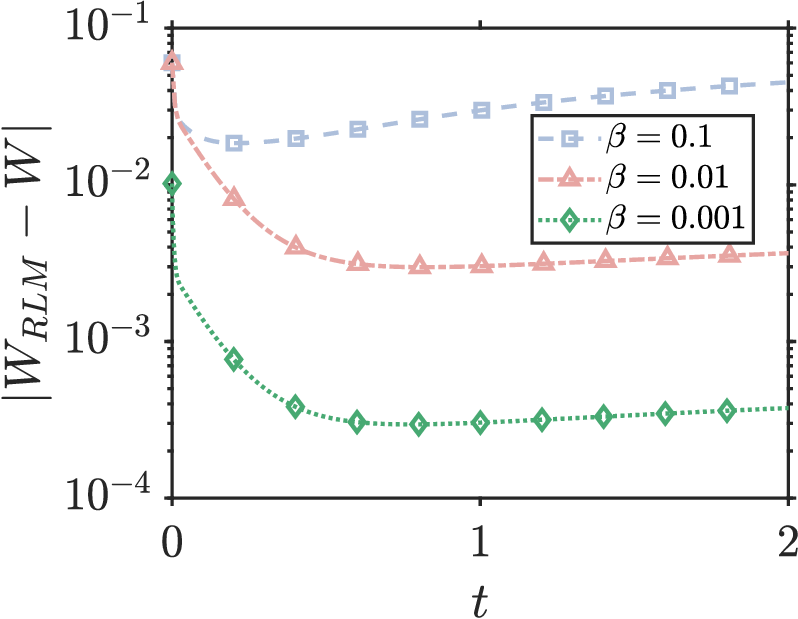}
		\hspace{0.02\textwidth}
		\includegraphics[width=0.3\linewidth]{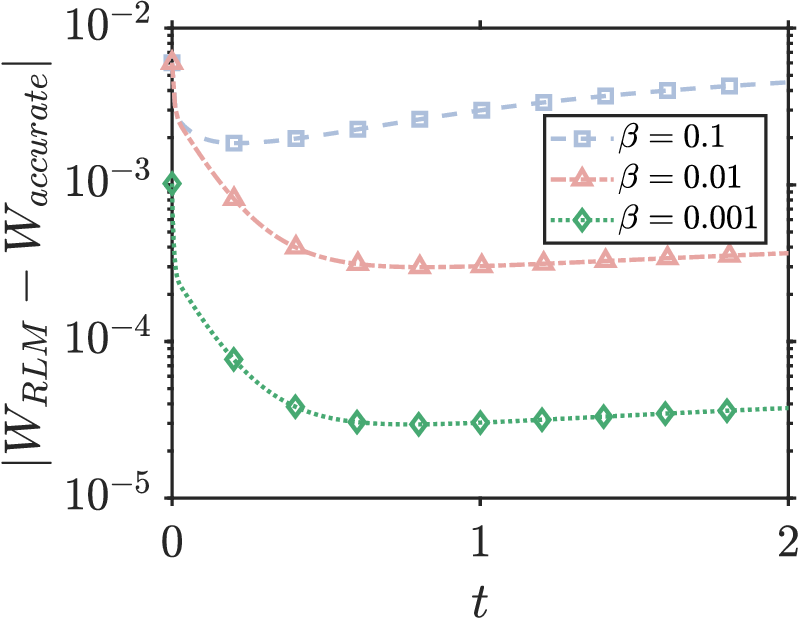}
		\caption{Temporal evolution of the Willmore energy for the A-RLM-BDF1-FDM. Left: modified energy $W_{\mathrm{RLM}}(t)$ compared with the accurate energy $W_{\mathrm{accurate}}(t)$ (inset: zoomed-in view near $t=2$). Middle: energy difference $|W_{\mathrm{RLM}} - W|$ between the modified and original energies. Right: absolute error $|W_{\mathrm{RLM}} - W_{\mathrm{accurate}}|$ relative to the accurate energy.}
		\label{fig:RLM_energy_bdf1}
	\end{figure}
	\begin{figure}
		\centering
		\includegraphics[width=0.3\linewidth]{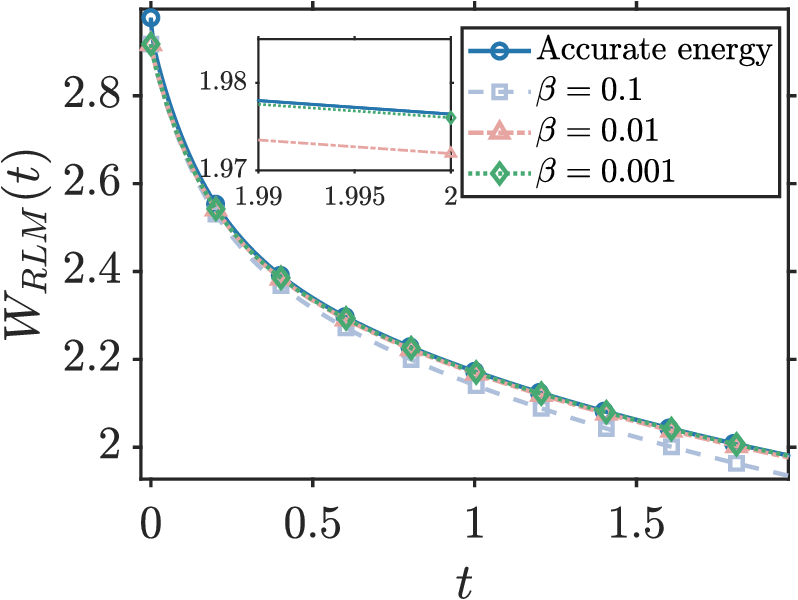}
		\hspace{0.02\textwidth}
		\includegraphics[width=0.3\linewidth]{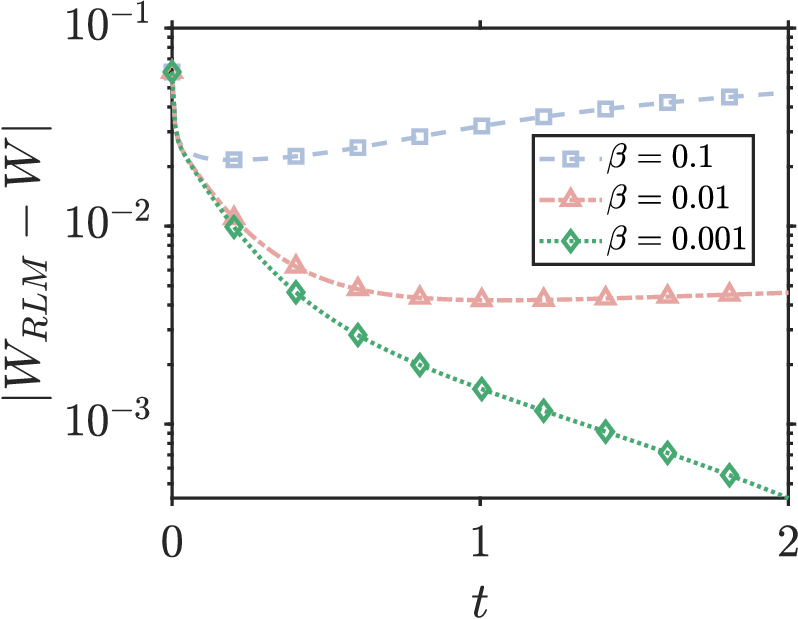}
		\hspace{0.02\textwidth}
		\includegraphics[width=0.3\linewidth]{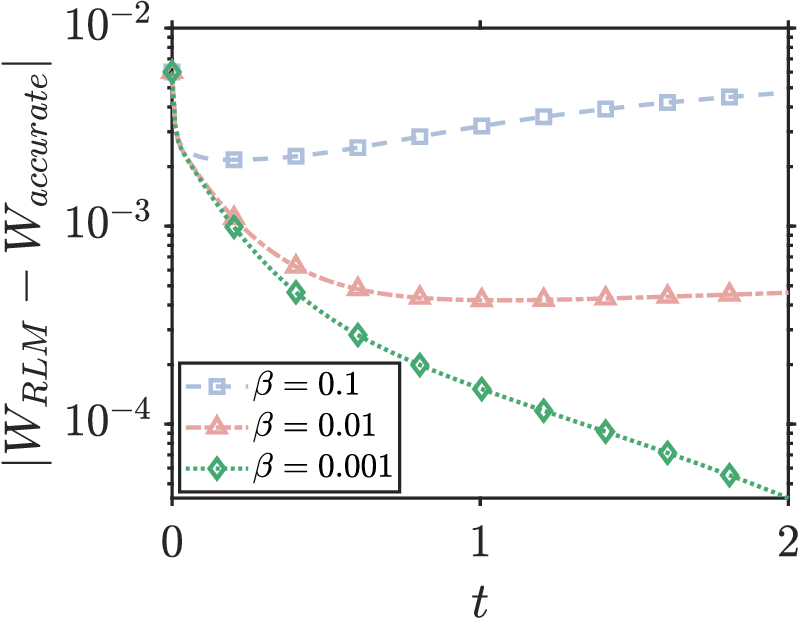}
		\caption{Temporal evolution of the Willmore energy for the A-RLM-BDF2-FDM. Left: modified energy $W_{\mathrm{RLM}}(t)$ compared with the accurate energy $W_{\mathrm{accurate}}(t)$ (inset: zoomed-in view near $t=2$). Middle: energy difference $|W_{\mathrm{RLM}} - W|$ between the modified and original energies. Right: absolute error $|W_{\mathrm{RLM}} - W_{\mathrm{accurate}}|$ relative to the accurate energy.}
		\label{fig:RLM_energy_bdf2}
	\end{figure}
\end{example}

\begin{example}(Evolution from simple to complex initial shapes)
We present the evolution of planar curves under the Willmore flow for a sequence of test cases with increasing geometric complexity, as shown in Fig.~\ref{fig:evolution}. The initial curves range from relatively simple, nearly convex shapes to highly nonconvex configurations featuring narrow necks and pronounced geometric oscillations.
For all considered initial curves, both the A-WAR algorithm and the A-BDF$k$-FDMs are able to compute the curve evolution over the entire time interval without numerical breakdown. As illustrated in the figure, the curves exhibit an approximately uniform expansion while their geometries become progressively smoother. Notably, even for the most intricate nonconvex initial shapes, the numerical evolution remains stable and free of spurious oscillations or loss of resolution.
The incorporation of adaptivity plays a crucial role in ensuring the robustness of the numerical scheme, allowing accurate resolution of geometric features and maintaining stability throughout the entire evolution process.

	\begin{figure}[!t]
		\centering
		\includegraphics[width=0.23\linewidth]{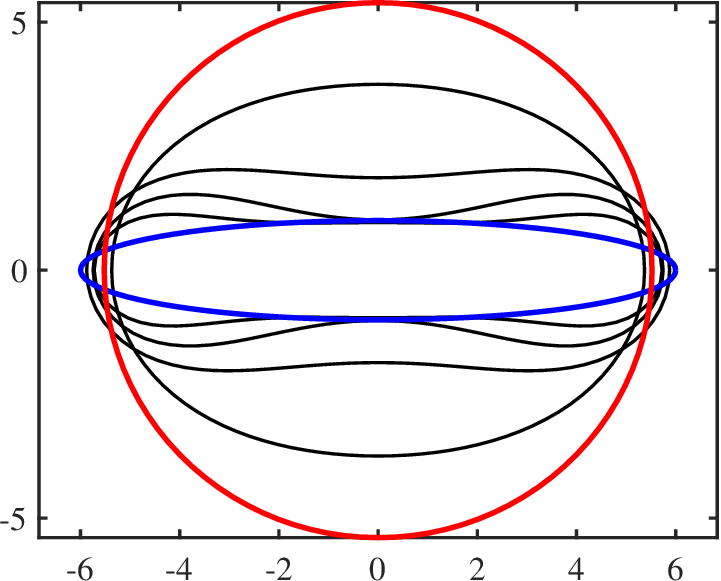} 
		\hspace{0.02\textwidth}
		\raisebox{8mm}{\includegraphics[width=0.25\linewidth]{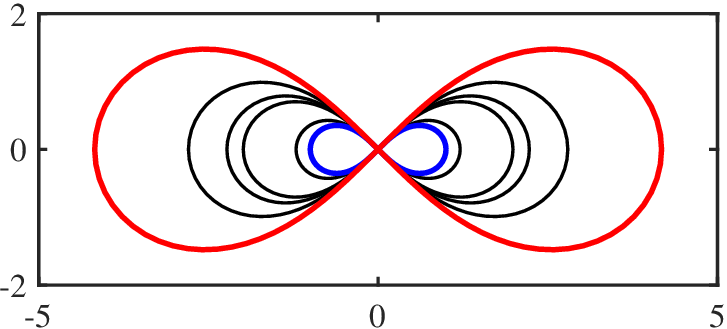}
			\hspace{0.02\textwidth}}
		\includegraphics[width=0.25\linewidth]{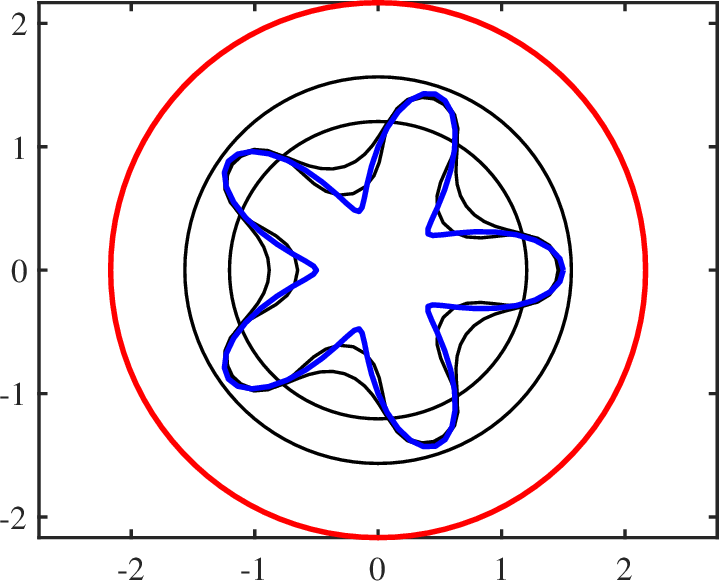}
		\\
		\includegraphics[width=0.25\linewidth]{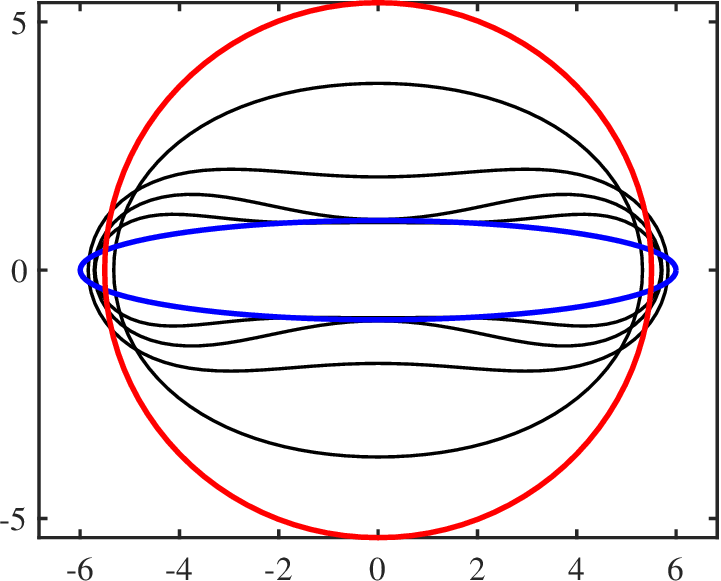}
		\hspace{0.02\textwidth}
		\raisebox{8mm}{\includegraphics[width=0.25\linewidth]{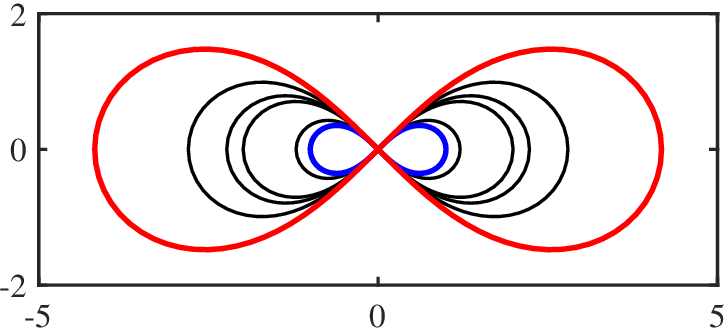}
			\hspace{0.02\textwidth}}
		\includegraphics[width=0.25\linewidth]{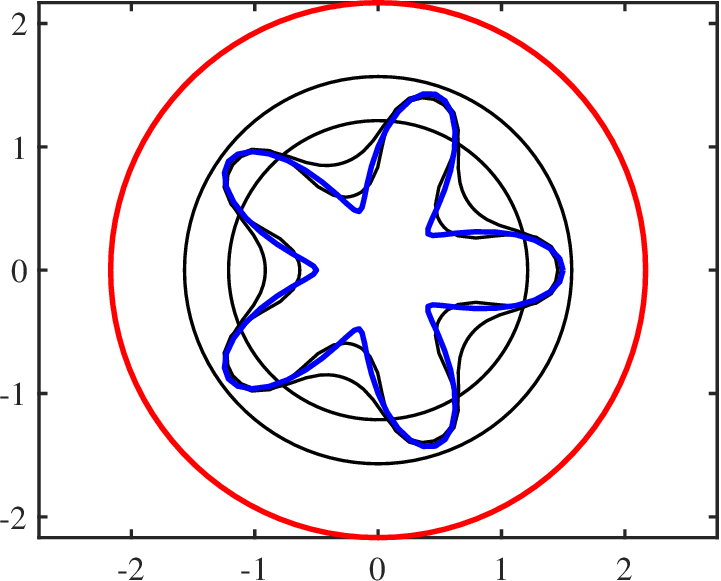}
		\caption{The first row corresponds to the A-WAR algorithm, while the second row shows the results obtained with the A-BDF$k$-FDMs (k=1). In each subplot, the initial curve is shown in blue, several intermediate states are plotted in black, and the final curve is displayed in red. Left column: the initial curve is given by $x(\rho)=6\cos(2\pi\rho)$ and $y(\rho)=\sin(2\pi\rho)$. The curves  are shown at times $t=0,1,5,20,80,300$. Middle column: the initial curve is given by $x=\frac{\cos(2\pi\rho)}{1+\sin^2(2\pi\rho)}, 
			y=\frac{\cos(2\pi\rho)\sin(2\pi\rho)}{1+\sin^2(2\pi\rho)}.$ The curves are shown at times $t=0,0.1,0.3,0.6,1.2,10$. Right column: the initial curve is given by $x=\left(1+0.5\cos(10\pi\rho)\right)\cos(2\pi\rho),
			y=\left(1+0.5\cos(10\pi\rho)\right)\sin(2\pi\rho).
			$The curves  are shown at times $t=0,0.002,0.015,0.05,2,10$.}
		\label{fig:evolution}
	\end{figure}
\end{example}

\section{Conclusions}\label{sec6}

In this paper, we develop adaptive moving mesh methods for the planar Willmore flow by incorporating a tangential velocity into the original geometric evolution equation. 
The tangential velocity is derived from the variational derivative of a mesh functional associated with a curvature-based monitor function, which enables dynamic redistribution of mesh points along the evolving interface. 
Based on this formulation, a fully adaptive moving mesh Willmore system is constructed, and its temporal and spatial discretizations lead to the A-BDF$k$-FDMs.
To achieve energy stability for the adaptive schemes, a RLM approach is introduced, resulting in the A-RLM-BDF$k$-FDMs. The proposed RLM-based schemes theoretically preserve the discrete Willmore energy stability law while retaining the adaptive moving mesh property. Moreover, additional adaptive strategies are developed to enhance the flexibility of the proposed framework.
Extensive numerical experiments demonstrate that the proposed methods effectively capture the evolution of interfaces with complex geometric features while maintaining high-quality meshes. The results also confirm that the RLM-based schemes provide accurate and energy-stable approximations for long-time simulations of the Willmore flow.
In future work, we will extend the proposed adaptive and energy-stable framework to more general two- and three-dimensional geometric flow problems and further develop efficient structure-preserving numerical methods with improved adaptivity and computational performance.

\appendix
\section{Choice of monitor function}
\label{app:monitor_fun_choice}
In this appendix, we provide a geometric motivation for the form of the monitor functions used in the main text by analyzing the local interpolation error of a
smooth curve approximated by a linear polygonal interpolant. 

To elucidate how local geometric features of an evolving curve influence interpolation accuracy, we consider a sufficiently smooth planar curve $\Gamma$ and focus on a local segment $\Gamma_i$ determined by two consecutive nodes $\mathbf{X}_i$ and $\mathbf{X}_{i+1}$. As illustrated in Fig.~\ref{fig:monitor_fun}, the local geometric interpolation error can be interpreted as the maximum normal deviation between the smooth curve and its linear polygonal interpolation. 
We denote the distance between the two consecutive nodes by $h_i = \left\|\mathbf{X}_{i+1} - \mathbf{X}_i\right\|.$ Since the normal distance between a curve and its linear interpolation is invariant under rigid motions, we may, without loss of generality, apply a local translation and rotation of coordinates. Specifically, the point $\mathbf{X}_i$ is mapped to the origin, and the direction of the linear interpolant connecting $\mathbf{X}_i$ and $\mathbf{X}_{i+1}$ is aligned with a reference axis. In the resulting local coordinate system, the curve segment $\Gamma_i$ can be represented in graph form as $\overline{\Gamma}(\overline{x}), ~\overline{x} \in [0,h_i],$ with the normalization $\overline{\Gamma}(0)=0$. Let $\overline{x}_* \in (0,h_i)$ denote a point at which the normal deviation between the curve and its linear interpolant attains its maximum. By the necessary condition for an extremum, we have $\overline{\Gamma}'(\overline{x}_*) = 0.$

Using the Taylor expansion of $\overline{\Gamma}(\overline{x})$ in the neighborhood of $\overline{x}=0$ and evaluating it at $\overline{x}=\overline{x}_*$, we obtain
\begin{equation}\label{eq:Gamma}
\overline{\Gamma}(\overline{x}_*)
=
\frac{\overline{x}_*^2}{2}\overline{\Gamma}''(\overline{x}_*)
-
\frac{\overline{x}_*^3}{6}\overline{\Gamma}'''(\overline{x}_*) + O(\overline{x}_*^4).
\end{equation}
We first examine the contribution of the second-order term. For a planar curve expressed in graph form, the curvature is given by
\[
\left|\kappa\right|
=
\frac{\left|\overline{\Gamma}''\right|}{\left[1+\left(\overline{\Gamma}'\right)^2\right]^{3/2}}.
\]
At the point $\overline{x}_*$, where $\overline{\Gamma}'(\overline{x}_*)=0$, this relation simplifies to
\[
\left|\overline{\Gamma}''(\overline{x}_*)\right| = \left|\kappa(\overline{x}_*)\right|.
\]
Consequently, the second-order term in the Taylor expansion contributes an error component proportional, in magnitude, to $\overline{x}_{*}^2|\kappa|$, indicating that in regions of mild curvature the local interpolation error is primarily curvature-driven.

We next consider the geometric interpretation of the third-order remainder term.
For convenience of description, we temporarily employ the arc-length parametrization and write the curve as $\mathbf{X}(s)$. The Frenet-Serret
relations for planar curves imply
\[
\mathbf{X}''(s) = -\kappa(s)\mathbf{n}(s), \qquad
\mathbf{X}'''(s) = -\partial_s \kappa(s)\mathbf{n}(s) + \kappa^2(s)\mat{\tau}(s),
\]
where $\mat{\tau}(s)$ and $\mathbf{n}(s)$ denote the unit tangent and normal vectors, respectively. It follows that the third derivative of $\mathbf{X}$ satisfies 
\[
\left|\mathbf{X}'''(s)\right|
\le
\left|\partial_s \kappa(s)\right| + \kappa^2(s).
\]
Under the assumption of locally small slopes, the third derivative $\overline{\Gamma}'''(\overline{x}_*)$ in the graph representation is of the same order as $\mathbf{X}'''(s)$ in the arc-length parametrization, and hence
\[
\left|\overline{\Gamma}'''(\overline{x}_*)\right|
\;\lesssim\;
\left|\partial_s \kappa(\overline{x}_*)\right| + \kappa^2(\overline{x}_*).
\]
Substituting this estimate into the \eqref{eq:Gamma} and using $\overline{x}_* \le h_i$, we obtain an upper bound for the local geometric interpolation error on the segment $\Gamma_i$:
\[
\begin{aligned}
	\max_{\overline{x}\in(0,h_i)} \left|\overline{\Gamma}(\overline{x})\right|
	&\lesssim
	h_i^2 \left|\kappa\right|
	+
	h_i^3 \left|\partial_s \kappa\right|
	+
	h_i^3 \kappa^2 \\
	&\lesssim
	C_1 h_i^2 \left|\kappa\right|
	+
	C_2 h_i^3 \left|\partial_s \kappa\right|
	+
	C_3 h_i^3 \kappa^2,
\end{aligned}
\]
where the constants $C_1$, $C_2$, and $C_3$ depend only on the regularity of the curve.


This estimate reveals a hierarchical structure in the local geometric interpolation error. Unlike \cite{Mackenzie2019}, where the monitor function is constructed solely in terms of the curvature $\kappa$ under an implicit asymptotic assumption $h_i \to 0$, we consider a practically relevant regime where $h_i$ is moderate (i.e., not strictly in the asymptotic limit $h_i \to 0$). In this non-asymptotic setting, higher-order geometric contributions cannot be automatically neglected, as different terms may enter the error expansion at comparable orders. Specifically, when the curvature variation is prominent such that $h_i \vert{}\partial_s \kappa\vert{} \approx \vert{}\kappa\vert{}$, the third-order contribution $h_i^3 \vert{}\partial_s \kappa\vert{}$ becomes comparable in magnitude to $h_i^2 \vert{}\kappa\vert{}$, and therefore cannot be neglected. Similarly, when $h_i \kappa^2\approx \kappa$, the curvature-squared term $h_i^3 \kappa^2$ also contributes non-negligibly to the local error. This analysis provides a sound theoretical justification for incorporating higher-order geometric information into the design of monitor functions.

\begin{figure}
	\centering
	\includegraphics[width=0.8\linewidth]{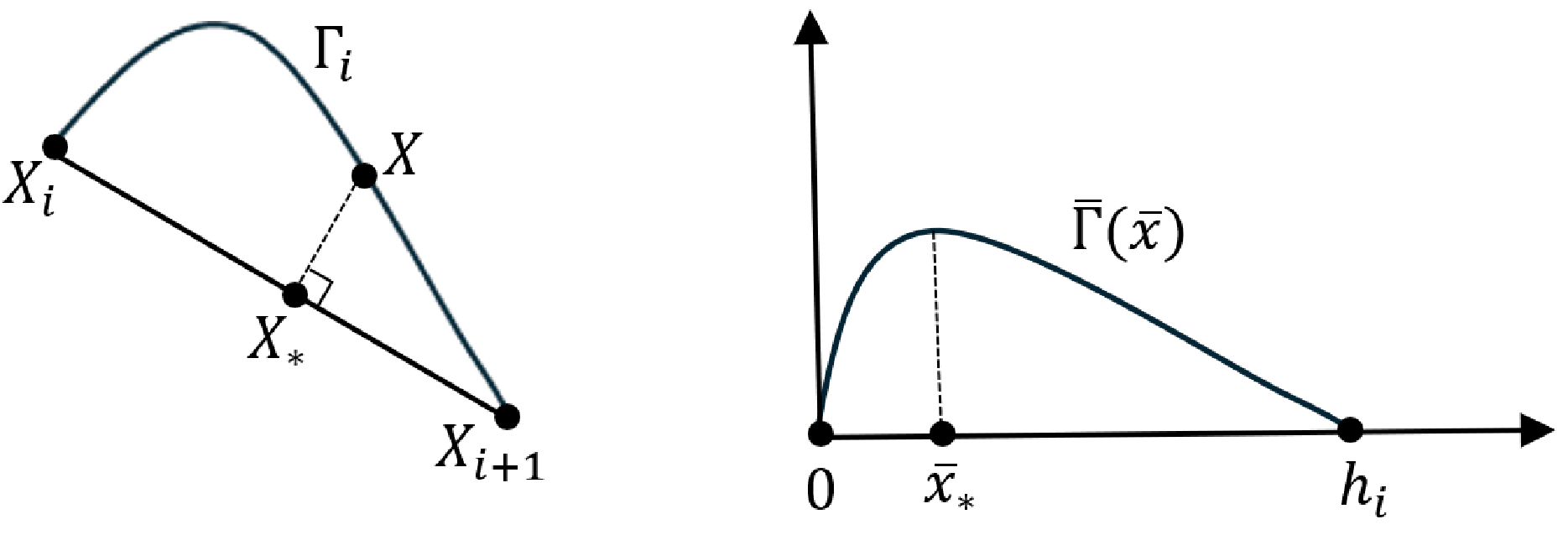}
	\caption{Left: A local segment $\Gamma_i$ of a smooth planar curve together with its piecewise linear interpolation between consecutive nodes $\mathbf{X}_i$ and $\mathbf{X}_{i+1}$. For a point $\mathbf{X}$ on the curve, the local interpolation error is measured by the normal distance to the corresponding point $\mathbf{X}_*$ on the linear interpolant. Right: After applying a rigid translation and rotation, the curve segment is represented in graph form as $\overline{\Gamma}(\overline{x})$ over the interval $[0,h_i]$. The maximal geometric interpolation error on $\Gamma_i$ corresponds to the absolute maximum of $\overline{\Gamma}(\overline{x})$, attained at $\overline{x}=\overline{x}_*$.}
	\label{fig:monitor_fun}
\end{figure}


\section{Supplementary weighted arc-length redistribution strategy}
\label{app:WAR}

In this appendix, to accommodate more complex initial curves, we present the weighted arc-length redistribution strategy employed in the numerical experiments. The redistribution is based on the weighted equidistribution principle and aims to improve the spatial distribution of mesh points throughout the evolution.
Specifically, for a discrete closed curve
\[
\Gamma_h=\{\mathbf X_i\}_{i=1}^{M},
\]
the redistributed mesh is generated by enforcing the weighted arc-length
equidistribution principle, namely,
\begin{equation}
	\int_0^{s_i}m(s)\,{\rm d}s
	=
	\frac{i}{M}
	\int_0^Lm(s)\,{\rm d}s,
	\qquad
	i=1,\ldots,M,
\end{equation}
where $L$ denotes the total length of the curve and
$m(s)$ is the monitor function that controls the mesh concentration.

The monitor function controls the density of mesh points according to
the local geometric characteristics of the curve. In above proposed adaptive moving mesh methods, the
monitor function is constructed from curvature-related quantities, i.e.,
\[
m
=
m(\kappa,\partial_s\kappa),
\]
where both the curvature and its arc-length variation are taken into
account. To adapt the redistribution strategy to different geometric
configurations, the monitor function is selected automatically according
to several geometric indicators. The corresponding adaptive procedure is
summarized in Algorithm~\ref{alg:monitor}. After the monitor function has been determined, the mesh is
redistributed according to the weighted equidistribution principle.
The detailed implementation is summarized in
Algorithm~\ref{alg:redistribution}.


\begin{algorithm}[t]
	\caption{Adaptive monitor selection strategy}
	\label{alg:monitor}
	\small
	
	\textbf{Input:}
	The nodal coordinates
	$\mathbf X_j:=\mathbf X_j^{n+1}$,
	$j=1,\ldots,M$,
	computed by the BDF$k$-FDMs.
	
	\vspace{0.3em}
	
	\noindent
	\textbf{Step 1: Compute geometric indicators.}
	
	Compute the segment lengths
	
	\[
	\Delta s_j
	=
	\|\mathbf X_{j+1}-\mathbf X_j\|,
	\]
	
	the discrete curvature derivative
	
	\[
	\delta_s\kappa_j
	=
	\frac{\kappa_{j+1}-\kappa_{j-1}}
	{\Delta s_j+\Delta s_{j-1}},
	\]
	
	and evaluate
	
	\[
	C_0=\max_j|\kappa_j|,
	\qquad
	C_1=\max_j|\delta_s\kappa_j|,
	\qquad
	Q=
	\frac{\max_j\Delta s_j}
	{\min_j\Delta s_j}.
	\]
	
	\vspace{0.3em}
	
	\noindent
	\textbf{Step 2: Construct the basic monitor using $C_0$.}
	
	According to the curvature magnitude,
	
	\[
	m_j=
	\begin{cases}
		1+\alpha |\kappa_j|,
		&
		C_0<C_0^{\rm low},
		\\[0.4em]
		
		1+\alpha
		\Bigl(
		(1-\beta)|\kappa_j|
		+
		\beta\kappa_j^2
		\Bigr),
		&
		C_0^{\rm low}
		\le
		C_0
		<
		C_0^{\rm high},
		\\[0.5em]
		
		1+\alpha\kappa_j^2,
		&
		C_0
		\ge
		C_0^{\rm high}.
	\end{cases}
	\]
	
	where
	$0<\beta<1$
	is the blending parameter.
	
	\vspace{0.3em}
	
	\noindent
	\textbf{Step 3: Incorporate curvature variation using $C_1$.}
	
	If
	
	\[
	C_1<C_1^{\rm low},
	\]
	
	retain the monitor obtained in Step~2.
	
	Otherwise,
	
	\[
	m_j=
	\begin{cases}
		m_j+\gamma|\delta_s\kappa_j|,
		&
		C_1^{\rm low}
		\le
		C_1
		<
		C_1^{\rm high},
		\\[0.5em]
		
		1+\alpha
		\sqrt{
			\kappa_j^2
			+
			\gamma
			(\delta_s\kappa_j)^2},
		&
		C_1
		\ge
		C_1^{\rm high}.
	\end{cases}
	\]
	
	Here
	$\alpha=\alpha_0$
	and
	$\gamma=\gamma_0$
	are prescribed parameters.
	
	\vspace{0.3em}
	
	\noindent
	\textbf{Step 4: Perform weighted arc-length redistribution.}
	
	Use the monitor values
	$\{m_j\}$
	together with the current nodal coordinates
	as the input of
	Algorithm~\ref{alg:redistribution}
	to obtain the redistributed nodes
	$\{\mathbf X_j^{\rm new}\}$.
	
	\vspace{0.3em}
	
	\textbf{Output:}
	The redistributed mesh
	$\{\mathbf X_j^{\rm new}\}$.
	
\end{algorithm}



\begin{algorithm}[htbp]
	\caption{Adaptive weighted arc-length redistribution (A-WAR) algorithm.}
	\label{alg:redistribution}
	\small
	
	\textbf{Input:}
	The discrete curve
	$\{\mathbf{X}_j^{\,n+1}\}_{j=0}^{M}$
	computed by the BDF$k$-FDMs,
	together with the monitor values
	$\{m_j\}_{j=0}^{M}$.
	
	\vspace{0.3em}
	\noindent$\bullet$~\textbf{Step 1: Compute segment lengths.}~~
	Compute the discrete arc-length increments
	\[
	\Delta s_j
	=
	\left\|\mathbf{X}_{j+1}-\mathbf{X}_{j}\right\|,
	\qquad
	j=0,\ldots,M-1.
	\]
	
	\vspace{0.3em}
	\noindent$\bullet$~\textbf{Step 2: Compute weighted arc-length increments.}~~
	Using the monitor values, evaluate
	\[
	\omega_j
	=
	\frac12(m_j+m_{j+1})\,\Delta s_j,
	\qquad
	j=0,\ldots,M-1.
	\]
	
	\vspace{0.3em}
	\noindent$\bullet$~\textbf{Step 3: Compute the cumulative weighted arc-length.}~~
	Construct the cumulative weighted arc-length distribution
	\[
	w_0=0,
	\qquad
	w_j
	=
	\sum_{k=0}^{j-1}\omega_k,
	\qquad
	j=1,\ldots,M,
	\]
	and define the total weighted arc-length by
	\[
	L_w=w_M.
	\]
	
	\vspace{0.3em}
	\noindent$\bullet$~\textbf{Step 4: Construct the target distribution.}~~
	Generate the equidistributed weighted arc-length coordinates
	\[
	w_i^{\mathrm{tar}}
	=
	\frac{i}{M}L_w,
	\qquad
	i=0,\ldots,M.
	\]
	
	\vspace{0.3em}
	\noindent$\bullet$~\textbf{Step 5: Redistribute the mesh nodes.}~~
	Determine the new parameter values
	$\rho_i^{\mathrm{new}}$
	by inverting the discrete mapping
	\[
	w_j \mapsto \rho_j
	\]
	through piecewise linear interpolation at the target values
	$w_i^{\mathrm{tar}}$.
	The redistributed mesh points are then obtained by evaluating the discrete
	curve at the new parameter values,
	\[
	\mathbf{X}_i^{\mathrm{new}}
	=
	\mathbf{X}(\rho_i^{\mathrm{new}}).
	\]
	
	\vspace{0.3em}
	\noindent$\bullet$~\textbf{Step 6: Update the mesh.}~~
	For closed curves, impose the periodicity condition
	\[
	\mathbf{X}_0^{\mathrm{new}}
	=
	\mathbf{X}_M^{\mathrm{new}}.
	\]
	Finally, update the computational mesh by
	\[
	\mathbf{X}_j
	\leftarrow
	\mathbf{X}_j^{\mathrm{new}},
	\]
	which is used as the computational mesh for the next BDF$k$-FDM time step.
	
\end{algorithm}


The above A-WAR strategy allocates
additional mesh points to regions with relatively large curvature or
rapid curvature variation while maintaining a smooth overall mesh
distribution. Numerical experiments demonstrate that the proposed
strategy effectively improves mesh quality and enhances the robustness
of the numerical simulation for evolving interfaces with complicated
geometric structures.

\begin{rem}
 Compared with the A-BDF$k$-FDMs, the A-WAR strategy may be capable of
 handling more complicated curve evolutions. However, the redistribution
 procedure is essentially a heuristic mesh adaptation technique, which is
 introduced externally rather than being derived from the underlying
 geometric evolution equations. Therefore, it is difficult to establish
 rigorous theoretical properties for the resulting numerical method.
 In particular, in contrast to the A-RLM-BDF$k$-FDMs, the A-WAR strategy
 does not possess a theoretical guarantee of discrete energy stability.
\end{rem}

\bibliographystyle{elsarticle-num}
\bibliography{thebib}

\end{document}